\documentclass[12pt,leqno,a4paper]{amsart}
\usepackage{amssymb}
\newcommand{\F}{{\mathbb{F}}}

\newcommand{\Q}{{\mathbb{Q}}}
\newcommand{\Z}{{\mathbb{Z}}}

\newcommand{\SG}{{\mathfrak{S}}}
\newcommand{\cc}{{\mathfrak{C}}}
\newcommand{\ab}{{\mathbf{a}}}
\newcommand{\bb}{{\mathbf{b}}}
\newcommand{\bC}{{\mathbf{C}}}

\newcommand{\bG}{{\mathbf{G}}}
\newcommand{\jrm}{{\mathbf{j}}}
\newcommand{\Jrm}{{\mathbf{J}}}
\newcommand{\Cd}{{C}}
\newcommand{\CC}{{\mathcal{C}}}

\newcommand{\cG}{{\mathcal{G}}}
\newcommand{\cH}{{\mathbf{H}}}
\newcommand{\cI}{{\mathcal{I}}}
\newcommand{\cM}{{\mathcal{M}}}

\newcommand{\tH}{{\tilde{\mathbf{H}}}}
\newcommand{\tJ}{{\tilde{\mathbf{J}}}}
\newcommand{\Ind}{{\operatorname{Ind}}}
\newcommand{\Irr}{{\operatorname{Irr}}}
\newcommand{\Uch}{{\operatorname{Uch}}}

\newcommand{\Aut}{{\operatorname{Aut}}}
\newcommand{\kott}{{\Upsilon}}
\newcommand{\tW}{{\tilde{W}}}
\newcommand{\tell}{{\tilde{\ell}}}
\newcommand{\tchi}{{\tilde{\chi}}}
\renewcommand{\leq}{\leqslant}
\renewcommand{\geq}{\geqslant}

\newtheorem{thm}{Theorem}[section]
\newtheorem{lem}[thm]{Lemma}
\newtheorem{conj}[thm]{Conjecture}
\newtheorem{cor}[thm]{Corollary}
\newtheorem{prop}[thm]{Proposition}

\theoremstyle{definition}
\newtheorem{exmp}[thm]{Example}
\newtheorem{defn}[thm]{Definition}
\newtheorem{abschnitt}[thm]{}

\theoremstyle{remark}
\newtheorem{rem}[thm]{Remark}

\begin{document}

\title{On Kottwitz' conjecture for twisted involutions}

\date{\today}

\author{Meinolf Geck}
\address{Institute of Mathematics, Aberdeen University,
  Aberdeen AB24 3UE, Scotland, UK.}
\email{m.geck@abdn.ac.uk}


\begin{abstract}
Kottwitz' conjecture is concerned with the intersections of Kazhdan--Lusztig
cells with conjugacy classes of involutions in finite Coxeter groups.
In joint work with Bonnaf\'e, we have recently found a way to prove this
conjecture for groups of type $B_n$ and $D_n$. The argument for type $D_n$ 
relies on two ingredients which were used there without proof: (1) a 
strengthened version of the ``branching rule'' and (2) the consideration of 
``$\diamond$-twisted'' involutions where $\diamond$ is a graph automorphism. 
In this paper we deal with (1), (2) and complete the argument for type $D_n$; 
moreover, we establish Kottwitz' conjecture for $\diamond$-twisted 
involutions in all cases where $\diamond$ is non-trivial.
\end{abstract}

\maketitle

\pagestyle{myheadings}
\markboth{Geck}{Kottwitz' conjecture for twisted involutions}

\section{Introduction} \label{sec:intro}

Let $W$ be a finite Coxeter group with generating $S$. We assume that
we have a map $w \mapsto w^\diamond$ which is a group automorphism of
$W$ such that $S^\diamond=S$ and $(w^\diamond)^\diamond =w$ for all
$w \in W$. An element $w \in W$ is called a $\diamond$-twisted involution 
if $w^\diamond =w^{-1}$.  Given such an element $w \in W$, Kottwitz 
\cite{kottwitz} defined a character $\kott_w^\diamond$ of $W$ which only 
depends on the $\diamond$-conjugacy class of $w$ and which is remarkable
for various reasons:
\begin{itemize}
\item[(1)] The decomposition of $\kott_w^\diamond$ into irreducible 
characters is related to Lusztig's Fourier transforms \cite[Chap.~4]{LuB} 
associated with the various $\diamond$-stable ``families'' of $\Irr(W)$. 
\item[(2)] By Lusztig and Vogan \cite{LV} there is a natural lift of 
$\kott_w^\diamond$ to the generic one-parameter Iwahori-Hecke algebra 
associated with $W,S$. (By \cite{Lu12a}, there is even a version for 
arbitrary Coxeter groups.)
\item[(3)] Kottwitz \cite{kottwitz} conjectures that, for any left cell 
$\Gamma$ of $W$ in the sense of Kazhdan--Lusztig \cite{KL}, the number of 
elements in the intersection of $\Gamma$ with the $\diamond$-conjugacy 
class of~$w$ equals the scalar product of $\kott_w^\diamond$ with 
the character afforded by $\Gamma$.
\end{itemize}
Following Kottwitz, we say that we are in the ``split'' case if 
$w^\diamond=w$ for all $w\in W$; otherwise, we are in the ``quasi-split'' 
case. If $W$ is irreducible, then the quasi-split cases to consider are as
follows:
\begin{itemize}
\item $W$ of type $A_n$, $E_6$ and $\diamond$ given by conjugation with 
the longest element in $W$;
\item $W$ of type $D_n$ and $\diamond$ the non-trivial graph automorphism
of order $2$;
\item $W$ of type $F_4$, $I_2(m)$ and $\diamond$ the non-trivial
graph automorphism.
\end{itemize}
The results in this paper combined with previous work by Casselman 
\cite{cass}, Kottwitz \cite{kottwitz}, Marberg \cite{marberg}, 
Bonnaf\'e and the author \cite{boge}, \cite{pycox} will show that both 
the ``split'' and the ``quasi--split'' case of the conjecture in (3) 
hold for all $W$ except possibly for type $E_8$. (A.~Halls at the 
University of Aberdeen is currently working on type $E_8$.)

In Section~\ref{sec:kott}, we introduce
Kottwitz' involution module, both the split and the quasi-split version. 
In (\ref{luvo}) we show that this coincides with the module constructed
by Lusztig and Vogan \cite{LV}. (The identification in the split
case already appeared in \cite[\S 2]{gema}.) We then also discuss various
examples: first of all, the case where $\diamond$ is given by conjugation 
with the longest element in $W$; furthermore, the cases where $W$ is of 
type $F_4$, $I_2(m)$ and $\diamond$ is the non-trivial graph automorphism. 

In Section~\ref{sec:dn}, which may be of independent interest, we clarify
some notoriously troublesome issues concerning those irreducible 
characters of a Coxeter group of type $D_n$ which are not invariant
under the graph automorphism of order~$2$. The main result is 
Proposition~\ref{solveD2} which establishes a strengthened version of 
``Pieri's Rule'' for these characters. This was used without proof in 
\cite{boge} to remove some ambiguities in the determination of the character
of the split version of Kottwitz' involution module for type $D_n$. 

In Section~\ref{sec:basic}, we consider Lusztig's leading coefficients
of character values of Iwahori--Hecke algebras. In \cite[Chap.~12]{LuB},
\cite{lulc}, Lusztig has used his classification of the unipotent
characters of a finite reductive group to determine the leading
coefficients in the split case. Here, we extend at least some
of these results to the quasi-split case. The main difficulty consists
in carefully choosing extensions of $\diamond$-invariant characters of $W$
to the semidirect product of $W$ with $\langle \diamond \rangle \subseteq 
\mbox{Aut}(W)$. The applications to the quasi-split case in type $D_n$ are 
contained in Proposition~\ref{invDn}.

Finally, in Section~\ref{sec:tDn}, we complete the proof of Kottwitz' 
conjecture for type $D_n$. The main idea is to treat the split and the 
quasi-split case at the same time. For this purpose, we develop a 
modified version of Kottwitz' conjecture for type $B_n$, where we consider 
the left cells with respect to a suitable weight function in the sense 
of Lusztig \cite{lusztig}. The main result in this section is 
Theorem~\ref{maint}. This involves the construction of a modified version 
of Kottwitz' involution module in Lemma~\ref{typeD2}. At first sight, 
this new combined setting seems to make things more complicated (which is 
certainly true from a technical point of view); but, in fact, I do 
not see any way how to carry out the argument separately for the split 
and the quasi-split case.

We shall use standard results and notation concerning the (complex)
characters of finite groups. If $\chi,\psi$ are two class functions
on a finite group $G$, then $\langle \chi,\psi\rangle_G$ denotes the
usual scalar product for which the irreducible characters of $G$ form
an orthonormal basis. If $H \subseteq G$ is a subgroup and $\psi$ is
a class function on $H$, then $\Ind_H^G(\psi)$ denotes the induced
class function on $G$. We denote by $\Irr(G)$ the set of all irreducible
characters of $G$.

\section{Kottwitz' twisted involution module} \label{sec:kott}

Let $W$ be a finite Coxeter group with generating set $S$. For $w \in W$, 
we denote by $\ell(w)$ the length of $w$. We assume that we have a map 
$w \mapsto w^\diamond$ which is an automorphism of $W$ such that 
$S^\diamond=S$ and $(w^\diamond)^\diamond=w$ for all $w\in W$. We say that 
two elements $w,w' \in W$ are  $\diamond$-conjugate if there exists some 
$x \in W$ such that $w'=x^\diamond wx^{-1}$. This defines an equivalence 
relation on $W$, and the corresponding equivalence classes will be called the 
$\diamond$-conjugacy classes of~$W$. The subgroup 
\[ C_W^\diamond(w):= \{ x\in W \mid x^\diamond w= w x\}\]
is called the $\diamond$-centraliser of $w$ in $W$.
Let $\Phi$ be the root system of $W$ and $\Phi=\Phi^+\amalg \Phi^-$ be the
partition into positive and negative roots determined by $S$. 
Since $S^\diamond=S$, the automorphism $w \mapsto w^\diamond$ defines a
permutation of the simple roots in $\Phi$. We shall assume that this
permutation induces a map $\alpha \mapsto \alpha^\diamond$ on all of 
$\Phi$ such that 
\[ w^\diamond(\alpha^\diamond)=w(\alpha)^\diamond \qquad \mbox{for all 
$w \in W$ and $\alpha \in \Phi$}.\]

\begin{defn} \label{kodef}
An element $w\in W$ is called a ``$\diamond$-twisted involution'' 
if $w^\diamond=w^{-1}$. Given such an element $w\in W$, let $\Phi_w$ be 
the set of all $\alpha \in \Phi$ such that $w(\alpha)=-\alpha^\diamond$. 
Then, by Kottwitz \cite[2.1, 4.2]{kottwitz}, we can define a linear character 
$\varepsilon_w\colon C_W^\diamond(w) \rightarrow \{\pm 1\}$ as 
follows.  For $x \in C_W^\diamond(w)$ we have $\varepsilon_w(x)=(-1)^k$ 
where $k$ is the number of positive roots $\alpha \in \Phi_w$ such that 
$x(\alpha)$ is negative. Then set 
\[ \kott_w^\diamond:=\Ind_{C_W^\diamond(w)}^W\bigl(\varepsilon_w\bigr).\]
\end{defn}

\begin{rem} \label{conjcl} Let $\Cd\subseteq W$ be any  subset which is a
union of $\diamond$-conjugacy classes of $\diamond$-twisted involutions in
$W$. If $\Cd$ is a single $\diamond$-conjugacy class, then we certainly have 
$\kott_w^\diamond=\kott_{w'}^\diamond$ for all $w,w'\in \Cd$. In general,
we set 
\[ \kott_\Cd^\diamond=\sum_w \kott_w^\diamond\]
where $w$ runs over a set of representatives of the $\diamond$-conjugacy 
classes contained in $\Cd$. In particular, this applies to the set of all
$\diamond$-twisted involutions in $W$.
\end{rem}

\begin{rem} \label{koco1} Assume that $w^\diamond=w$ for all $w \in W$;
we just write this as $\diamond=1$. Then a $\diamond$-twisted involution
$w$ is just an ordinary involution in $W$. Furthermore, $\kott_w^\diamond$ 
is the character of the involution module in the ``split case''; see 
\cite[\S 2]{kottwitz}. So, here, this character will be denoted by 
$\kott_{w}^1$.
\end{rem}

If $\diamond$ is non-trivial, then Kottwitz \cite{kottwitz}  formulated 
Definition~\ref{kodef} in a slightly different way, using the semidirect 
product of $W$ with the automorphism given by~$\diamond$. The two versions 
are equivalent by the following remark.

\begin{rem} \label{semidir} Let $\tW$ be the semidirect product of $W$ 
with the subgroup $\langle \diamond \rangle \subseteq \Aut(W)$. Thus, 
$\tilde{W}$ is generated by $W$ and an additional element $\gamma$ such 
that $\gamma w \gamma= w^\diamond$ for all $w \in W$. (If $\diamond=1$, 
then $\gamma=1$ and $\tW=W$; otherwise, $\gamma^2=1$.) The natural action 
of $W$ on $\Phi$ can be extended to $\tW$ such that $\gamma(\alpha)=
\alpha^\diamond$ for all $\alpha \in \Phi$.  

First of all, this
shows that $w\in W$ is a $\diamond$-twisted involution if and only if 
$\gamma w$ is an ordinary involution in $\tilde{W}$. 
Furthermore, two elements $w,w' \in W$ are $\diamond$-conjugate if and only 
if $\gamma w, \gamma w'$ are conjugate in~$\tilde{W}$. Consequently, the 
map $w \mapsto \gamma w$ defines a bijection between the $\diamond$-conjugacy classes of $W$ and the ordinary conjugacy classes of 
$\tilde{W}$ which are contained in the coset $\gamma W\subseteq \tilde{W}$.
We have $C_{\tilde{W}}(\gamma w)=C_W^\diamond(w)$ for all $w \in W$.
\end{rem}

\begin{rem} \label{minl} For any subset $I \subseteq S$ we denote by 
$W_I \subseteq W$ the corresponding parabolic subgroup and by $w_I$ the 
longest element in $W_I$. Let $\Cd$ be a $\diamond$-conjugacy class 
of $\diamond$-twisted involutions. Then there exists a subset $I \subseteq S$
such that $w_I \in \Cd$ and $s^\diamond w_I=w_Is$ for all $s \in I$; 
furthermore, $I^\diamond=I$ and $w_I$ has minimum length in $\Cd$. 
(See \cite[Prop.~3.2.10]{gepf} for the case where $\diamond$ is the identity 
and He \cite[Lemma~3.6]{he1} for the general case.) If we take $w=w_I$,
then one easily sees that the set of roots $\Phi_{w_I}$ is just the 
parabolic subsystem $\Phi_I \subseteq \Phi$ corresponding to~$I$. Let 
$\Pi_I\subseteq \Phi_I^+$ be the set of simple roots. Then we have:
\begin{equation*}
\mbox{$W_I$ is a normal subgroup of $C_W^\diamond(w_I)$}.  \tag{a}
\end{equation*}
Indeed, since $s^\diamond w_I=w_Is$ for all $s \in I$, we certainly have 
$W_I \subseteq C_W^\diamond(w_I)$. Now let $x \in C_W^\diamond(w_I)$ and 
$\alpha\in\Phi_I$. Then $w_I(x(\alpha))=x^\diamond (w_I(\alpha))=-x^\diamond
(\alpha^\diamond)=-x(\alpha)^\diamond$ and so $x(\alpha) \in \Phi_{w_I}=
\Phi_I$, as required. Thus, (a) is proved. Consequently, by  Howlett
\cite[Cor.~3]{How}, we have a semidirect product decomposition
\begin{equation*}
C_W^\diamond(w_I)=Y \ltimes W_I \qquad \mbox{where} \qquad Y:=\{y \in 
C_W^\diamond(w_I) \mid y(\Pi_I)=\Pi_I\}.\tag{b} 
\end{equation*}
Note that $Y$ is contained in the set of distinguished left coset 
representatives of $W_I$ in $W$; in particular, each element of $Y$ 
sends all positive roots in $\Phi_I$ to positive roots. We conclude that
\begin{equation*}
\varepsilon_{w_I}(yw)=(-1)^{\ell(w)} \qquad \mbox{for all $y \in Y$ and
$w \in W_I$}.\tag{c}
\end{equation*}
This provides an explicit description of $\Upsilon_{w_I}^\diamond$ which
will be useful in several places below.
\end{rem}

\begin{abschnitt} \label{luvo} Let $\Cd$ be a $\diamond$-conjugacy class
of $\diamond$-twisted involutions in $W$. Let $M$ be a $\Q$-vector space 
with a basis $\{a_w \mid w\in \Cd\}$. By Lusztig and Vogan \cite[7.1]{LV} 
(see also \cite{Lu12a}), it is known that $M$ is a $\Q[W]$-module, where 
a generator $s \in S$ acts via the following formula: 
\[ s.a_w=\left\{\begin{array}{cl} -a_w  & \qquad \mbox{if $s^\diamond w=ws$
and $\ell(ws)<\ell(w)$},\\ a_{s^\diamond ws} & \qquad \mbox{otherwise}.
\end{array}\right.\]
Let $w \in C$ be fixed. As discussed in \cite[6.3]{LV}, we obtain a group 
homomorphism 
\[ \eta_w \colon C_W^\diamond(w) \rightarrow \{\pm 1\}\]
such that $x.a_w=\eta_w(x)a_w$ for all $x \in C_W^\diamond(w)$;
furthermore, $M$ is isomorphic to the $\Q[W]$-module obtained by inducing 
$\eta_w$ from $C_W^\diamond(w)$ to $W$. We claim that 
\begin{center}
\fbox{$\Upsilon_w^\diamond$ is the character afforded by the 
$\Q[W]$-module $M$.}
\end{center}
If $\diamond$ is trivial, this is shown in \cite[\S 2]{gema}. The argument
in the general case is similar. Indeed, let $I \subseteq S$ and $w_I \in 
\Cd$ be as in Remark~\ref{minl}. Then we have $C_W^\diamond(w_I)=Y 
\ltimes W_I$. Thus, it will be sufficient to show that 
\[ \eta_{w_I}(yw)=(-1)^{\ell(w)} \qquad \mbox{for all $y \in Y$
and $w \in W_I$}.\]
We argue as in \cite[Lemma~2.1]{gema}. For  $s \in I$, we have $s.a_{w_I}=
-a_{w_I}$ and so $\eta_{w_I}(s)=-1$. Consequently, we have 
$\eta_{w_I}(w)=(-1)^{\ell(w)}$ for all $w \in W_I$. Thus, it 
remains to show that $y.a_{w_I}=a_{w_I}$ for all $y \in Y$. We shall in 
fact show that $x.a_{w_I}=a_{x^\diamond w_Ix^{-1}}$ where $x$ is any
distinguished left coset representative of $W_I$ in $W$. We proceed by 
induction on $\ell(x)$. If $x=1$, the assertion is clear. Now assume that 
$x\neq 1$ and choose $s \in S$ such that $\ell(sx)<\ell(x)$. By 
Deodhar's Lemma \cite[2.1.2]{gepf}, we also have that $z:=sx$ is a 
distinguished left coset representative. Hence, using induction, we have 
$z.a_{w_I}=a_{z^\diamond w_Iz^{-1}}$ and so
\[ x.a_{w_I}=s.a_{z^\diamond w_Iz^{-1}}.\]
Let $u=z^\diamond w_Iz^{-1}$. Given the formula for the action of a 
generator on the basis elements of $M$, it now suffices to show that 
either $s^\diamond u\neq us$ or that $\ell(us)>\ell(u)$. Assume, if possible,
that none of these two conditions is satisfied, that is, we have 
$s^\diamond u=us$ and $\ell(us)<\ell(u)$; in particular, $\ell(szw_I
(z^\diamond)^{-1})=\ell(su^{-1})<\ell(u^{-1})=\ell(zw_I(z^\diamond)^{-1})$. 
But then the ``Exchange Lemma'' (see \cite[Exc.~1.6]{gepf})  and the fact 
that $\ell(szw_I)=\ell(z)+\ell(w_I)+1$ imply that $szw_I(z^\diamond)^{-1}=
zw_Iz'$ where $\ell(z')<\ell(z)$. Since $s^\diamond u=us$, we have 
$zw_I(z^\diamond)^{-1}s^\diamond=szw_I(z^\diamond)^{-1}=zw_Iz'$ and so 
$(z^\diamond)^{-1}s^\diamond=z'$. This would imply that $\ell(z')=
\ell((z^\diamond)^{-1}s^\diamond)=\ell((sz)^\diamond)=\ell(sz)>\ell(z)$, a 
contradiction. Hence, the assumption was wrong and so $x.a_{w_I}=
a_{x^\diamond w_Ix^{-1}}$, as required. Thus, the above claim is proved.
\end{abschnitt}


Next, we briefly recall the definition of Kazhdan--Lusztig cell modules.

\begin{abschnitt} \label{defkl}
Let $\cH$ be the generic one-parameter Iwahori--Hecke algebra associated
with $(W,S)$, over the ring of Laurent polynomials $A={\Z}[v,v^{-1}]$ in an 
indeterminate~$v$. Thus, $\cH$ has a basis $\{T_w \mid w \in W\}$ and, for 
any $s \in S$ and $w \in W$, the multiplication is given by 
\[ T_sT_w=\left\{\begin{array}{cl} T_{sw} & \qquad \mbox{if $\ell(sw)>
\ell(w)$},\\ T_{sw}+(v-v^{-1})T_w & \qquad \mbox{if $\ell(sw)<\ell(w)$}.
\end{array}\right.\]
Let $\{\bC_w \mid w \in W\}$ be the Kazhdan--Lusztig basis of $\cH$. For
any $w \in W$, we have 
\[ \bC_w=T_w+\sum_{y \in W, y<w} (-1)^{\ell(w)+\ell(y)}v^{l(w)-l(y)}P_{y,w}
(v^{-1}) T_y,\]
where $P_{y,w}\in \Z[v]$ are the polynomials defined in 
\cite[Theorem~1.1]{KL} and $y<w$ denotes the Bruhat--Chevalley order. 
We write
\[ \bC_x\bC_y=\sum_{z \in W} h_{x,y,z}\bC_z \qquad \mbox{where 
$h_{x,y,z}\in A$ for all $x,y,z \in W$}.\]
Let $\leq_L$ be the pre-order relation on $W$ defined in \cite{KL}; for
any $w \in W$, we have
\[ \cH \bC_w \subseteq \sum_{y \in W, y \leq_L w} A \bC_y.\]
For $y,w\in W$, we write $y \sim_L w$ if $y \leq_L w$ and $w \leq_L y$.
This defines an equivalence relation on $W$; the equivalence classes are
called the {\em left cells} of $W$. Let $\Gamma$ be such a left cell. 
Let $[\Gamma]_A$ be a free $A$-module with a basis $\{e_x \mid x \in 
\Gamma\}$. By the definition of $\sim_L$, this is an $\cH$-module where
the action is given by
\[ \bC_x.e_y=\sum_{z \in \Gamma} h_{x,y,z} e_z \qquad \mbox{for all $x\in W$
and $y \in \Gamma$}.\]
Then we obtain a $\Q[W]$-module $[\Gamma]_1$ by extension of scalars
via the unique ring homomorphism $A \rightarrow \Q$ such that $v \mapsto 1$.
We identify $[\Gamma]_1$ with its character. 
\end{abschnitt}

\begin{conj}[Kottwitz \protect{\cite{kottwitz}}] \label{koco} 
Let $w\in W$ be a $\diamond$-twisted involution and $\Cd$ be its 
$\diamond$-conjugacy class in $W$. Let $\Gamma$ be a left cell in $W$. Then
\begin{center}
\fbox{$\langle \kott_w^\diamond,[\Gamma]_1\rangle_{W}=|\Cd \cap \Gamma|.$}
\end{center}
\end{conj}

Note that, by Remark~\ref{semidir}, the above formulation is indeed 
equivalent to the original formulation by Kottwitz. The above formulation
covers both the case where $\diamond$ is the identity (``split'' case)
and the case where $\diamond$ is non-trivial (``quasi-split'' case).

\begin{rem} \label{dperm} The map $\diamond$ induces an $A$-algebra
automorphism $h \mapsto h^\diamond$ on $\cH$ such that $T_w^\diamond=
T_{w^\diamond}$ for all $w \in W$. One easily checks that $\bC_w^\diamond=
\bC_{w^\diamond}$ for all $w \in W$. Consequently, $\diamond$ permutes
the left cells of $W$, the right cells of $W$ and the two-sided
cells of $W$. 

Now let $\Gamma \subseteq \cc$ be a left cell of $W$ and $\Cd$ be a 
$\diamond$-conjugacy class of $\diamond$-twisted involutions in $W$. Then 
we claim that 
\[ \Cd \cap \Gamma =\varnothing \qquad \mbox{unless} \qquad 
\cc^\diamond=\cc.  \]
Indeed, assume that $\cc^\diamond \neq \cc$ and that there exists some
$w \in \Cd \cap \Gamma$. Then $w^{-1}=w^\diamond \in \Gamma^\diamond 
\subseteq \cc^\diamond$ and, hence, $w^{-1} \not\in \cc$. But this 
contradicts \cite[Lemma~5.2(iii)]{LuB} which shows that, for any $w \in \cc$,
we also have $w^{-1} \in \cc$. Thus, the above claim is proved.

This statement provides a first test case for Conjecture~\ref{koco}: 
we will have to show that then we also have $\langle \kott_w^\diamond, 
[\Gamma]_1\rangle_W=0$ for $w \in \Cd$; see Remark~\ref{dperm1}
below for type $D_n$. 
\end{rem}

\begin{exmp} \label{longest} Let $w_0 \in W$ be the longest element
and assume that $w_0$ is not central in $W$. Then define $w^\diamond=
w_0ww_0$ for $w \in W$. One easily checks that the following hold.
\begin{itemize}
\item[(a)] $w\in W$ is a $\diamond$-twisted involution if and and only
if $w_0w$ is an involution.
\item[(b)] Let $w\in $W be a $\diamond$-twisted involution and 
$\Cd$ its $\diamond$-conjugacy class. Then $C_W^\diamond(w)=
C_W(w_0w)$ and $\Cd=w_0\CC$ where $\CC$ is the ordinary conjugacy 
class of $w_0w$ in $W$.
\end{itemize}
Now let $w\in W$ be a $\diamond$-twisted involution and $\Gamma$ be a 
left cell in $W$. Then $\Gamma w_0$ and $w_0 \Gamma$ are also left cells 
and we have $[w_0 \Gamma]_1=[\Gamma w_0]_1=[\Gamma]_1 \otimes \varepsilon$; 
see \cite[5.14]{LuB}, \cite[11.7]{lusztig}. Consequently, we obtain
\[ |\Cd\cap \Gamma|=|w_0(\CC \cap w_0\Gamma)|=|\CC \cap w_0\Gamma|,\]
where $\Cd$ is the $\diamond$-conjugacy class of $w$ and $\CC$ is
the ordinary conjugacy class of $w_0w$ in $W$. On the other hand, by 
\cite[5.3.1]{kottwitz}, we have
\[\langle \kott_{w_0w}^1, \chi \otimes \varepsilon\rangle_W=
\langle \kott_w^\diamond, \chi\rangle_W \qquad \mbox{for all 
$\chi \in \Irr(W)$}.\]
This yields that 
\[ \langle \kott_w^\diamond, [\Gamma]_1\rangle_W=\langle \kott_{w_0w}^1, 
[\Gamma]_1 \otimes \varepsilon\rangle_W=\langle \kott_{w_0w}^1, 
[w_0\Gamma]_1 \rangle_W.\]
So, if the split version of Kottwitz' conjecture holds for $W$, then we
have 
\[ \langle \kott_w^\diamond, [\Gamma]_1\rangle_W=\langle \kott_{w_0w}^1, 
[w_0\Gamma]_1 \rangle_W=|\CC \cap w_0\Gamma |=|\Cd \cap \Gamma|,\]
that is, the quasi-split version (with respect to $\diamond$) also holds. 

This discussion applies, in particular, to $(W,S)$ of type $A_n$, $D_{2n+1}$
$E_6$, $I_2(2m+1)$. The split version of Kottwitz' conjecture holds in type 
$A_n$, as already observed by Kottwitz himself \cite{kottwitz}; see also
\cite[Exp.~4.10]{boge}. For type $D_{2n+1}$, see \cite[Cor.~7.6]{boge} and 
Corollary~\ref{mainc} below. Finally, Casselman \cite{cass} has verified 
that the split version holds in type $E_6$; see Marberg \cite{marberg} for 
the dihedral groups. 
\end{exmp}

\begin{exmp} \label{qi2} Let $m \geq 2$ and $(W,S)$ be of type $I_2(2m)$
where $S=\{s_1,s_2\}$ and $s_1s_2$ has order $2m$. Assume that 
$s_1^\diamond=s_2$ and $s_2^\diamond=s_1$. By Remark~\ref{minl}, it is 
clear that, up to $\diamond$-conjugacy, $1$ is the only $\diamond$-twisted
involution; let $\Cd_1$ denote its $\diamond$-conjugacy class. Let us 
consider the corresponding character $\kott_1^\diamond$. We have
\[ \Irr(W)=\{1,\varepsilon,\varepsilon_1,\varepsilon_2,\chi_1,\chi_2,
\ldots,\chi_{m-1}\}\]
where $1$ is the trivial character, $\varepsilon$ is the sign character,
$\varepsilon_1$ and $\varepsilon_2$ are two further characters of degree~$1$
and each $\chi_j$ has degree~$2$; see \cite[5.3.4]{gepf}. Here, the notation
is such that $\varepsilon_1(s_1)=\varepsilon_2(s_2)=1$ and 
$\varepsilon_1(s_2)=\varepsilon_2(s_1)=-1$; furthermore, $\chi_j$ is
determined by the condition that $\chi_j(1_2)=2\cos(\pi j/m)$. Now note that  
\[ C_W^\diamond(1)=\{x \in W \mid x^\diamond=x\}=\{1,1_{2m}\}.\]
Furthermore, $\Phi_1=\varnothing$ and so $\varepsilon_1$ is the trivial 
character of $C_W^\diamond(1)$. Then we find that 
\[ \kott_1^\diamond=\Ind_{C_W^\diamond(1)}^W(1)=\left\{\begin{array}{cl}
\displaystyle 1+\varepsilon+\varepsilon_1+\varepsilon_2+
\sum_{1\leq j\leq (m-2)/2} 2\chi_{2j} & \qquad \mbox{if $m$ is even},\\
\displaystyle  1+\varepsilon+ \sum_{1\leq j\leq (m-1)/2} 2\chi_{2j} & 
\qquad \mbox{if $m$ is odd}.  \end{array}\right.\]
Next, we consider the left cells of $W$. To simplify notation, write 
$1_k=s_1s_2s_1 \cdots$ ($k$ factors) and $2_k=s_2s_1s_2 \cdots$ ($k$ 
factors); in particular, $1_{2m}=2_{2m}$ is the longest element in $W$. 
By \cite[8.8]{lusztig}, the left cells are 
\begin{gather*}
\Gamma_0:=\{1_0\}, \qquad \Gamma_1:= \{1_1,2_2,1_3,\ldots,1_{2m-1}\},\\
\Gamma_2:=\{2_1,1_2,2_3,\ldots, 2_{2m-1}\}, \qquad\Gamma_{2m}:=\{1_{2m}\}.
\end{gather*} 
Thus, we have 
\[ |\Cd_1 \cap \Gamma_i|=\left\{\begin{array}{cl} 1 &\qquad \mbox{if $i=0$ 
or $i=2m$},\\ m-1 &\qquad \mbox{if $i=1$ or $i=2$}.\end{array}\right.\]
The characters of the left cell modules are given by:
\[ [\Gamma_0]_1=1, \quad [\Gamma_1]_1=\varepsilon_1+\sum_{1\leq j
\leq m-1} \chi_j, \quad [\Gamma_2]_1=\varepsilon_2+\sum_{1\leq j \leq
m-1}\chi_j, \quad [\Gamma_{2m}]_1=\varepsilon;\]
see, for example, \cite[Exp.~2.2.8]{geja}. Consequently, we have
\[\langle\kott_1^\diamond,[\Gamma_i]_1\rangle_W=\left\{\begin{array}{cl}
1&\qquad\mbox{if $i=0$ or $i=2m$},\\ m-1 & \qquad \mbox{if $i=1$ or $i=2$}.
\end{array}\right.\]
Hence, we see that Kottwitz' conjecture holds in this case.
\end{exmp}

\begin{exmp} \label{qf4} Let $(W,S)$ be of type $F_4$, where $S=\{s_0, 
s_1,s_2, s_3\}$ is such that $s_0s_1$ and $s_2s_3$ have order~$3$ and 
$s_1s_2$ has order~$4$. Assume that $s_0^\diamond=s_3$, $s_1^\diamond=s_2$,
$s_2^\diamond=s_1$ and $s_3^\diamond=s_0$. Let $\Cd_1=\{x^\diamond x^{-1}
\mid x \in W\}$ be the $\diamond$-conjugacy class containing $w=1$. Using 
Remark~\ref{minl} we find that $\Cd_1$ is the only $\diamond$-conjugacy 
class of $\diamond$-twisted involutions in $W$. We have $C_W^\diamond(1)=
\{x \in W \mid x^\diamond=x\}$; this is a dihedral group of order $16$, 
generated by $s_0s_3$ and $(s_1s_2)^2$. Since $\Phi_1=\varnothing$, we obtain 
by a direct computation (which can be done by hand):
\[ \kott_1^\diamond=\Ind_{C_W^\diamond(1)}^W(1)=1_1+1_4+2_1+2_2+2_3+2_4+
2\cdot 4_1+9_1+9_2+9_3+9_4+6_1+12_1,\]
where we use the notation for $\Irr(W)$ as in \cite[Table~C.3 (p.~413)]{gepf}.
Using a computer algebra system capable of computing Kazhdan--Lusztig cells,
it is straightfoward to check that Conjecture~\ref{koco} holds. For example,
using {\sf PyCox} \cite{pycox}, the left cells of $W$ and the characters 
of the corresponding left cell modules are obtained by the 
following commands: 
\begin{verbatim}
    >>> W=coxeter("F",4); l=klcells(W,1,v)[0]
    >>> ch=[leftcellleadingcoeffs(W,1,v,c)['char'] for c in l]
    >>> chartable(W)['charnames']
\end{verbatim}
The last command gives the labelling of $\Irr(W)$. The set $C_1$ is 
obtained by:
\begin{verbatim}
    >>> p=[3,2,1,0]     # the permutation on S={0,1,2,3}
    >>> C1=noduplicates([W.reducedword(w[::-1]+[p[s] for s in w],W) 
                                                for w in allwords(W)]
\end{verbatim}
Further explanations are available through the online help in {\sf PyCox}.
\end{exmp}

\begin{rem} \label{dchar} The map $w \mapsto w^\diamond$ on $W$ also 
induces an operation on $\Irr(W)$, which we denote by $\chi \mapsto 
\chi^\diamond$. We have $\chi^\diamond(w)=\chi(w^\diamond)$ for all 
$w \in W$. Let $\tilde{W}=\langle W,\gamma\rangle$ be the semidirect 
product as in Remark~\ref{semidir}. By standard results on Clifford theory,
we have
\[ \Ind_W^{\tW}(\chi) \in \Irr(\tW) \qquad \mbox{for all $\chi \in
\Irr(W)$ such that $\chi^\diamond\neq \chi$}.\]
On the other hand, let us denote
\[\Irr^\diamond(W):=\{\chi \in \Irr(W) \mid \chi^\diamond=\chi\}.\]
Then each $\chi \in \Irr^\diamond(W)$ has exactly two extensions to $\tW$, 
which differ only by a sign on elements in the coset $\gamma W$. 
\end{rem}

\begin{rem} \label{dfam} Let $\cc$ be a two-sided cell of $W$. We set
\[ \Irr(W\mid \cc)=\{\chi \in \Irr(W) \mid \langle [\Gamma]_1,\chi\rangle_W
\neq 0 \mbox{ for some left cell $\Gamma \subseteq \cc$}\}.\]
Alternatively, we have $\Irr(W\mid \cc):=\{\chi \in \Irr(W) \mid \chi 
\sim_{LR} w \mbox{ for some $w \in \cc$}\}$, where the relation  ``$\chi 
\sim_{LR} w$'' is defined in \cite[5.1 (p.~139)]{LuB}. (This easily follows 
from the definitions; see, for example, \cite[2.2.18]{LuB}.) Thus, we 
obtain a partition
\[ \Irr(W)=\coprod_{\cc} \Irr(W\mid \cc).\]
where $\cc$ runs over all two-sided cells in $W$. Now recall from 
Remark~\ref{dperm} that $\diamond$ permutes the left cells of $W$. One
easily sees that, for a left cell $\Gamma$ of $W$, we have
\[ \mbox{trace}(T_w,[\Gamma^\diamond]_A)=\mbox{trace}(T_{w^\diamond},
[\Gamma]_A) \qquad \mbox{for all $w \in W$}.\]
This certainly implies that $\Irr(W\mid \cc^\diamond)=\{ \chi^\diamond 
\mid \chi \in \Irr(W\mid \cc)\}$.
\end{rem}

\begin{rem} \label{exthec0} Having dealt with type $F_4$ and the
dihedral groups, we shall assume from now on that $W$ is a Weyl group and
that $\diamond$ is ``ordinary'' in the sense of \cite[3.1]{LuB}, that is,
whenever $s,t \in S$ are in the same $\diamond$-orbit, then the product
$st$ has order $2$ or $3$. This has the following consequences.
\begin{itemize}
\item[(a)] Each $\chi \in \Irr(W)$ can be realised over $\Q$. (This is a 
well-known fact; see, for example, \cite[6.3.8]{gepf}.) 
\item[(b)] The two extensions of any $\chi \in \Irr^\diamond(W)$ to $\tW$ 
can also be realised over $\Q$. (See \cite[Prop.~3.2]{LuB}). 
\item[(c)] If $\cc$ is a two-sided cell of $W$ such that $\cc^\diamond=\cc$,
then $\Irr(W\mid \cc)\subseteq \Irr^\diamond(W)$. (See \cite[4.17]{LuB}.)
\end{itemize}
In any case, as far as the quasi-split version of Kottwitz' conjecture is 
concerned, it now remains to deal with type $D_n$ and the non-trivial
graph automorphism of order~$2$. 
\end{rem}

\section{On the irreducible characters in type $D_n$} \label{sec:dn}

In this section we fix some notation concerning the irreducible characters 
of Coxeter groups of classical type. This is especially relevant in type 
$D_n$ for $n$ even, where there are characters which are not invariant 
under the graph automorphism of order~$2$; it will be important for us to 
know exactly how to distinguish these characters from each other. In 
Corollary~\ref{solveD2a}, we establish a strengthened ``branching rule'' 
for type $D_n$. Then, in (\ref{kottspecial}) and Proposition~\ref{kottD1},
we state explicit formulae for the decomposition of $\kott_w^1$ and 
$\kott_w^\diamond$.

\begin{abschnitt} \label{macd} For $\chi \in\Irr(W)$, let $\bb_\chi$ denote 
the smallest integer $i \geq 0$ such that $\chi$ occurs in the $i$th 
symmetric power of the standard reflection representation of $W$. For 
example, if $\varepsilon$ is the sign character of $W$, then 
\[ \bb_{\varepsilon}=|T|=\ell(w_0)\]
where $w_0\in W$ is the longest element and $T=\{wsw^{-1} \mid s \in S, 
w \in W\}$ is the set of all reflections in $W$ (see 
\cite[5.3.1(a)]{gepf}). Let $W'\subseteq W$ be a subgroup generated by 
reflections and let $T'=W'\cap T$. Let $\varepsilon'$ be the sign character 
of $W'$. By a result due to Macdonald (see \cite[5.2.11]{gepf}), there is 
a unique $\chi\in \Irr(W)$ such that $\bb_\chi=|T'|$ and 
\[ \Ind_{W'}^W(\varepsilon')=\chi+ \mbox{combination of various $\psi \in
\Irr(W)$ such that $\bb_\psi>\bb_\chi$}.\]
We shall denote this character by 
\[ \chi:=\jrm_{W'}^W(\varepsilon').\]
This ``$\jrm$-induction'' can be used to systematically construct all the 
irreducible characters of $W$ of type $A_{n-1}$, $B_n$ and $D_n$; 
see \cite[Chap.~5]{gepf}. 
\end{abschnitt}


\begin{exmp} \label{typeA} Let $n \geq 1$ and $W=\SG_n$ be the symmetric
group, where the generators are given by the basic transpositions $s_i=
(i+1)$ for $1 \leq i \leq n-1$. (We also set $\SG_0=\{1\}$.) It is 
well-known that the irreducible characters of $\SG_n$ are parametrized 
by the partitions of $n$; we write this as
\[ \Irr(\SG_n)=\{\chi^\alpha \mid \alpha \vdash n\}.\]
This labelling is determined as follows; see, for example, 
\cite[5.4.7]{gepf}. Given a partition $\alpha\vdash n$, we denote by 
$\alpha^*$ denote the transpose partition. Let $\SG_{\alpha^*} \subseteq 
\SG_n$ be the corresponding Young subgroup; we have $\SG_{\alpha^*} \cong 
\SG_{\alpha_1^*} \times \SG_{\alpha_2^*} \times \ldots \times
\SG_{\alpha_k^*}$, where $\alpha_1^*, \alpha_2^*,\ldots,\alpha_k^*$ are 
the parts of $\alpha^*$. Let $\varepsilon_{\alpha^*}$ be the sign 
character of $\SG_{\alpha^*}$. Then 
\[\chi^{\alpha}=\jrm_{\SG_{\alpha^*}}^{\SG_n}(\varepsilon_{\alpha^*}) 
\qquad \mbox{and} \qquad  \bb_{\chi^{\alpha}}=n(\alpha):=\sum_{1\leq i 
\leq l} (i-1)\alpha_i\]
where $\alpha=(\alpha_1\geq \alpha_2 \geq \ldots \geq \alpha_l\geq 0)$.
\end{exmp}

\begin{exmp} \label{typeB} Let $n \geq 1$ and $\tW_n$ be a Coxeter group 
of type $B_n$, with generators $\{t,s_1,s_2,\dots,s_{n-1}\}$ and diagram 
given as follows.
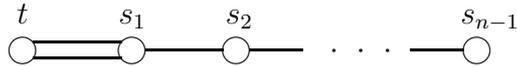
\begin{center}
\begin{picture}(220,30)
\put( 40, 10){\circle{10}}
\put( 44,  7){\line(1,0){33}}
\put( 44, 13){\line(1,0){33}}
\put( 81, 10){\circle{10}}
\put( 86, 10){\line(1,0){29}}
\put(120, 10){\circle{10}}
\put(125, 10){\line(1,0){20}}
\put(155,  7){$\cdot$}
\put(165,  7){$\cdot$}
\put(175,  7){$\cdot$}
\put(185, 10){\line(1,0){20}}
\put(210, 10){\circle{10}}
\put( 38, 20){$t$}
\put( 76, 20){$s_1$}
\put(116, 20){$s_2$}
\put(204, 20){$s_{n{-}1}$}
\end{picture}
\end{center}
(We also set $\tW_0=\{1\}$.) The irreducible characters of $\tW_n$ are 
parametrised by pairs of partitions $(\alpha,\beta)$ such that $|\alpha|+
|\beta|=n$. We write this as
\[ \Irr(\tW_n)=\{\tchi^{(\alpha,\beta)} \mid (\alpha,\beta) \vdash n\}.\]
For $(\alpha,\beta)\vdash n$, there is a reflection subgroup 
$\tW_{\alpha,\beta} \subseteq \tW_n$ of type 
\[ D_{\alpha_1} \times D_{\alpha_2} \times \ldots \times D_l \times
B_{\beta_1} \times B_{\beta_2} \times \ldots \times B_{\beta_k}\]
where $\alpha_1,\alpha_2,\ldots,\alpha_l$ are the parts of $\alpha$ and 
$\beta_1,\beta_2,\ldots,\beta_k$ are the parts of $\beta$. Let 
$\tilde{\varepsilon}_{\alpha,\beta}$ be the sign character on 
$\tW_{\alpha,\beta}$. Then, by \cite[5.5.1, 5.5.3]{gepf}, we have 
\[\tchi^{(\alpha,\beta)}=\jrm_{\tW_{\alpha,\beta}}^{\tW_n}
(\tilde{\varepsilon}_{\alpha, \beta}) \qquad \mbox{and} \qquad 
\bb_{\tchi^{(\alpha,\beta)}}=2n(\alpha)+ 2n(\beta)+|\beta|.\] 
Note also that $\tW_n\cong (\Z/2\Z)^n \rtimes \SG_n$ and there is a 
corresponding description of $\Irr(\tW_n)$ in terms of Clifford theory; see 
\cite[5.5.6]{gepf}.
\end{exmp}

\begin{exmp} \label{typeD} Let $n \geq 2$ and $W_n$ be a Coxeter group 
of type $D_n$, with generators $u,s_1,\ldots,s_{n-1}$ and diagram given 
as follows.
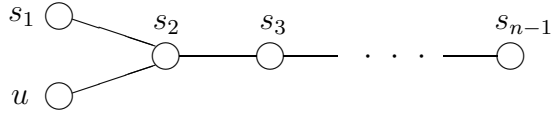
\begin{center}
\begin{picture}(220,40)
\put( 41,  3){\circle{10}}
\put( 41, 33){\circle{10}}
\put( 46,  4){\line(3,1){31}}
\put( 46, 32){\line(3,-1){31}}
\put( 81, 18){\circle{10}}
\put( 86, 18){\line(1,0){29}}
\put(120, 18){\circle{10}}
\put(125, 18){\line(1,0){20}}
\put(155, 15){$\cdot$}
\put(165, 15){$\cdot$}
\put(175, 15){$\cdot$}
\put(185, 18){\line(1,0){20}}
\put(210, 18){\circle{10}}
\put( 22, 31){$s_1$}
\put( 23,  0){$u$}
\put( 76, 28){$s_2$}
\put(116, 28){$s_3$}
\put(204, 28){$s_{n{-}1}$}
\end{picture}
\end{center}
Let $w \mapsto w^\diamond$ 
be defined by $u^\diamond=s_1$, $s_1^\diamond=u$ and $s_i^\diamond=s_i$ for 
$2 \leq i \leq n-1$. Then we can identify the semidirect product $W_n 
\rtimes \langle \diamond\rangle$ (see Remark~\ref{semidir}) with a 
Coxeter group $\tW_n$ of type $B_n$, with generators $t,s_1,\ldots,s_{n-1}$ 
and diagram as in Example~\ref{typeB}. We have an embedding
$W_n \hookrightarrow \tW_n$ given by the map
\[ u \mapsto ts_1t, \quad s_1 \mapsto s_1, \quad s_2 \mapsto s_2, 
\quad \ldots, \quad s_{n-1} \mapsto s_{n-1}.\]
Under this identification, we have $w^\diamond=twt$ for all $w \in W_n$.
(Thus, the generator $t$ is the ``additional'' element denoted by $\gamma$
in Remark~\ref{semidir}; by convention, we also set $W_0=W_1=\{1\}$,
where $\tW_0=\{1\}$ and $\tW_1=\{1,t\}$.) This provides a convenient 
setting for classifying the irreducible characters of $W_n$. Given 
$(\alpha, \beta) \vdash n$, we denote by $\chi^{[\alpha,\beta]}$ the 
restriction of $\tchi^{(\alpha, \beta)} \in \Irr(\tW_n)$ to $W_n$. Then 
we have (see \cite[5.6.1, 5.6.2]{gepf}):
\begin{itemize}
\item[(a)] If $\alpha \neq \beta$, then $\chi^{[\alpha,\beta]}=\chi^{[\beta,
\alpha]} \in \Irr(W_n)$. We have 
\[ \bb_{\chi^{[\alpha,\beta]}}=2n(\alpha)+2n(\beta)+
\min \{|\alpha|,|\beta|\}.\]
\item[(b)] If $\alpha=\beta$, then $\chi^{[\alpha,\beta]}=
\chi^{[\alpha,+]} +\chi^{[\alpha,-]}$ where $\chi^{[\alpha,+]}$, 
$\chi^{[\alpha,-]}$ are distinct irreducible characters of $W_n$. We have
\[ \bb_{\chi^{[\alpha,+]}}=\bb_{\chi^{[\alpha,-]}}=4n(\alpha)+n/2.\]
\end{itemize}
Furthermore, all irreducible characters of $W_n$ arise in this way. Of 
course, the second case can only occur if $n$ is even. In this case, the two 
characters $\chi^{[\alpha,\pm]}$ are explicitly given as follows;
see \cite[4.6.2]{LuB}. Let 
\[ H_n^+=\langle s_1,s_2,\ldots,s_{n-1}\rangle \qquad \mbox{and} \qquad 
H_n^-=\langle u,s_2,\ldots,s_{n-1}\rangle.\]
Both $H_n^+,H_n^-$ are isomorphic to $\SG_n$. Let $\alpha \vdash n/2$ and 
$\SG_{2\alpha^*}$ be the corresponding Young subgroup in $\SG_{n}$ where 
$2\alpha^*$ denotes the partition of $n$ obtained by multiplying all 
parts of $\alpha$ by $2$. We have corresponding subgroups $H_{2\alpha^*}^+
\subseteq H_n^+$ and $H_{2\alpha^*}^-\subseteq H_n^-$. Then
\[ \chi^{[\alpha,+]}=\jrm_{H_{2\alpha^*}^+}^{W_n}(\varepsilon_{2\alpha^*}^+)
\qquad \mbox{and} \qquad \chi^{[\alpha,-]}=\jrm_{H_{2\alpha^*}^-}^{W_n}
(\varepsilon_{2\alpha^*}^-)\]
where $\varepsilon_{2\alpha^*}^+$ denotes the sign character of
$H_{2\alpha^*}^+$ and $\varepsilon_{2\alpha^*}^-$ denotes the sign 
character of $H_{2\alpha^*}^-$. (This is also discussed in 
\cite[\S 5.6]{gepf} but \cite[5.6.3]{gepf} has to be reformulated as above.)  
\end{exmp}

We take this occasion to correct an error in \cite{gepf}. (This will
actually be essential for the proof of the strengthened ``branching rule''
in Corollary~\ref{solveD2a}.) Let $\varepsilon$ be the sign character 
of $W_n$. In \cite[Rem.~5.6.5]{gepf}, it is stated that $\chi^{[\alpha,+]}
\otimes \varepsilon=\chi^{[\alpha^*,+]}$, where $\alpha^*$ denotes the 
conjugate partition. This can be easily seen to be wrong already in small 
examples. The correct statement is as follows.

\begin{lem} \label{solveD1} Assume that $n\geq 2$ is even and let
$\alpha \vdash n/2$.
\begin{itemize}
\item[(a)] If $n/2$ is even, then $\chi^{[\alpha,+]}\otimes \varepsilon=
\chi^{[\alpha^*,+]}$ and $\chi^{[\alpha,-]}\otimes \varepsilon=
\chi^{[\alpha^*-]}$.
\item[(b)] If $n/2$ is odd, then 
\[ \chi^{[\alpha,+]}\otimes \varepsilon=\chi^{[\alpha^*,-]} \qquad \mbox{and}
\qquad \chi^{[\alpha,-]}\otimes\varepsilon=\chi^{[\alpha^*+]}.\]
\item[(c)] Let $\sigma_{n/2}:=s_1s_3 \cdots s_{n-1} \in W_n$. (Note that
$\sigma_{n/2}$ is a product of $n/2$ pairwise commuting generators $s_i$.)
Then 
\[ \chi^{[\alpha,+]}(\sigma_{n/2})-\chi^{[\alpha,-]}
(\sigma_{n/2})=(-1)^{n/2}\, 2^{n/2} \,\chi^{\alpha}(1)\]
where $\chi^{\alpha}$ denotes the irreducible character of $\SG_{n/2}$ 
labelled by $\alpha$.
\end{itemize}
\end{lem}

\begin{proof} In \cite[10.4.6]{gepf} (see also \cite[Thm.~5.1]{goetz1}), 
we find the definition of a collection of irreducible characters of $W_n$, 
which we denote here by $\{\psi^{[\alpha, \pm]} \mid \alpha \vdash n/2\}$, 
such that
\[\chi^{[\alpha,+]}+\chi^{[\alpha,-]}=\psi^{[\alpha,+]}+\psi^{[\alpha,-]}
\qquad \mbox{for all $\alpha \vdash n/2$};\]
furthermore, it is shown there that 
\[ \psi^{[\alpha,+]}(\sigma_{n/2})-\psi^{[\alpha,-]}
(\sigma_{n/2})=2^{n/2}\, \chi^{\alpha}(1).\]
Note that this identity allows us to distinguish $\psi^{[\alpha,+]},
\psi^{[\alpha,-]}$ one from another. Tensoring with $\varepsilon$, we obtain 
\[ (\psi^{[\alpha,+]} \otimes \varepsilon)(\sigma_{n/2})-(\psi^{[\alpha,-]}
\otimes \varepsilon) (\sigma_{n/2})=\varepsilon(\sigma_{n/2})\,2^{n/2}\, 
\chi^{\alpha}(1)=(-1)^{n/2}\, 2^{n/2} \,\chi^{\alpha}(1).\]
Now, by \cite[5.5.6]{gepf}, we have $\tchi^{(\alpha,\alpha)} \otimes
\tilde{\varepsilon}=\tchi^{(\alpha^*,\alpha^*)}$ where $\tilde{\varepsilon}$ 
is the sign character of $\tW_n$ (an extension of $\varepsilon$). This 
already implies that 
\[ \psi^{[\alpha,+]} \otimes \varepsilon=\psi^{[\alpha^*,\pm]} \qquad 
\mbox{and} \qquad \psi^{[\alpha,-]} \otimes \varepsilon=\psi^{[\alpha^*,
\mp]}.\]
Comparing with the identity 
\[ \psi^{[\alpha^*,+]}(\sigma_{n/2})-\psi^{[\alpha^*,-]}
(\sigma_{n/2})=2^{n/2}\, \chi^{\alpha^*}(1)=2^{n/2}\, \chi^{\alpha^*}(1),\]
we conclude that the desired description of the effect of tensoring with 
$\varepsilon$ holds for the characters $\psi^{[\alpha,\pm]}$. Once this is
shown, we can proceed as follows. By the computation in \cite[10.4.10]{gepf}, 
we have 
\[ \Big\langle \Ind_{H_{2\alpha^*}^+}^{W_n}(1_{2\alpha^*}), 
\psi^{[\alpha^*,+]}-\psi^{[\alpha^*,-]} \Big\rangle_{W_n}=1,\]
where $1_{2\alpha^*}$ stands for the trivial character of $H_{2\alpha^*}^+$. 
Consequently, we also have 
\[ \Big\langle \Ind_{H_{2\alpha^*}^+}^{W_n}(\varepsilon_{2\alpha^*}),
\psi^{[\alpha^*,+]}\otimes \varepsilon-\psi^{[\alpha^*,-]} \otimes 
\varepsilon' \Big\rangle_{W_n}=1,\]
where $\varepsilon_{2\alpha^*}$ is the sign character of $H_{2\alpha^*}^+$.
Comparing with the definition of $\chi^{[\alpha,\pm]}$ in Example~\ref{typeD},
we conclude that we must have 
\[\chi^{[\alpha,+]}=\psi^{[\alpha^*,+]} \otimes \varepsilon \qquad 
\mbox{and} \qquad \chi^{[\alpha,-]}=\psi^{[\alpha^*,-]}\otimes\varepsilon\]
for all $\alpha \vdash n/2$. This yields (a), (b), (c).
\end{proof}

\begin{rem} \label{speciale} Assume that $n \geq 2$ is even and let
$\sigma_{n/2}=s_1s_3 \cdots s_{n-1}\in W_n$ as above. Let $C_0$ be the
conjugacy class of $\sigma_{n/2}$ in $W_n$. By \cite[Prop.~3.4.12]{gepf}, 
we have $tC_0t \neq C_0$ and $\{C_0,tC_0t\}$ is the only pair 
of conjugacy classes of involutions with this property. Furthermore, a 
direct computation shows (see also the formula in \cite[4.3]{goetz1}):
\[ |C_{\SG_n}(\sigma_{n/2})|=2^{n/2}(n/2)! \qquad \mbox{and} \qquad 
|C_{W_n}(\sigma_{n/2})|= 2^n (n/2)!.\]
\end{rem}

We can now  state the following strengthening of the induction 
formula in \cite[6.4.9]{gepf}.

\begin{prop} \label{solveD2} Assume that $n \geq 2$ is even. Let $r \in 
\{2,4,\ldots,n\}$ and consider the parabolic subgroup $W'=W_{n-r}\times 
H_r$ where $W_{n-r}=\langle u,s_1,\ldots,s_{n-r-1}\rangle$ (type 
$D_{n-r}$) and $H_r=\langle s_{n-r+1},\ldots,s_{n-1} \rangle \cong \SG_r$. 
Let $\alpha' \vdash (n-r)/2$ and denote by $\varepsilon_r$ the sign 
character on the factor $H_r$. Then 
\[ \Ind_{W'}^{W_n}\bigl(\chi^{[\alpha',+]} \boxtimes \varepsilon_r\bigr)=
\sum_{\alpha} \chi^{[\alpha,+]} \quad+\quad  \mbox{``further terms''},\]
where the sum runs over all partitions $\alpha \vdash n/2$ whose 
Young diagram is obtained from that of $\alpha'$ by adding $r/2$ boxes,
with no two boxes in the same row; the expression ``further terms'' 
stands for a sum of various $\chi \in \Irr(W_n)$ which can be extended 
to $\tW_n$. 
\end{prop}

\begin{proof} By \cite[Prop.~6.4.9]{gepf} (and its proof), we already know
that 
\begin{align*}
\Ind_{W'}^{W_n}\bigl(\chi^{[\alpha',+]} \boxtimes \varepsilon_r\bigr)&=
\sum_{\alpha} \chi^{[\alpha,\mu_\alpha]}\quad+\quad\mbox{``further terms''},
\\\Ind_{W'}^{W_n}\bigl(\chi^{[\alpha',-]} \boxtimes \varepsilon_r\bigr)&=
\sum_{\alpha} \chi^{[\alpha,-\mu_\alpha]}\quad+\quad\mbox{``further terms''},
\end{align*}
where the sums run over all partitions $\alpha \vdash n/2$ as above and
where $\mu_\alpha\in \{\pm 1\}$. So it remains to determine the signs
$\mu_\alpha$. First note that the two ``further terms'' must be equal
since the above induced characters are conjugate to each other under~$t$. 
Hence, we have 
\[ \Ind_{W'}^{W_n}\bigl(\chi^{[\alpha',+]} \otimes \varepsilon_r\bigr)-
\Ind_{W'}^{W_n}\bigl(\chi^{[\alpha',-]} \otimes \varepsilon_r\bigr)=
\sum_{\alpha} \bigl(\chi^{[\alpha,\mu_\alpha]}-
\chi^{[\alpha,-\mu_\alpha]} \bigr)\] 
where the sum runs over all partitions $\alpha \vdash n/2$ as above. To
determine the signs, we evaluate both sides of this identity on the special 
element $\sigma_{n/2}$. Let $C_0$ denote the conjugacy class of 
$\sigma_{n/2}$. We have $\sigma_{n/2} \in W'$ and so $C_0 \cap W'\neq
\varnothing$; furthermore, $C_0 \cap W'$ can only contain elements $w \in W'$ 
such that $w,twt$ are not conjugate (see Remark~\ref{speciale}). 
Consequently, $C_0\cap W'$ is just the conjugacy class of $W'$ containing 
$\sigma_{n/2}$. Now note that 
\[ \sigma_{n/2}=\sigma_{(n-r)/2} \times s_{n-r+1} s_{n-r+3} \cdots 
s_{n-1} \in W'=W_{n-2}\times H_r.\]
This yields 
\[ (\chi^{[\alpha',\pm]} \boxtimes \varepsilon_r)
(\sigma_{n/2})= \chi^{[\alpha',\pm]}(\sigma_{(n-r)/2}) 
\varepsilon_r(s_{n-r+1}s_{n-r+3} \cdots s_{n-1})= (-1)^{r/2}
\chi^{[\alpha',\pm]} (\sigma_{(n-r)/2})\]
and so 
\[ \Ind_{W'}^{W_n}\bigl(\chi^{[\alpha',\pm]}\boxtimes \varepsilon_r\bigr)
(\sigma_{n/2})=(-1)^{r/2}\,\frac{|C_{W_n}(\sigma_{n/2})|}{|C_{W'}
(\sigma_{n/2})|} \chi^{[\alpha',\pm]}(\sigma_{(n-r)/2}).\]
Furthermore, by the formulae in Remark~\ref{speciale}, we have 
\[\frac{|C_{W_n}(\sigma_{n/2})|}{|C_{W'}(\sigma_{(n-2)/2})|}=[\SG_{n/2}:
\SG_{(n-r)/2} \times \SG_{r/2}] \, 2^{r/2}.\]
Thus, using also Lemma~\ref{solveD1} (applied to $W_{n-r}$), we conclude
that 
\begin{align*}
\Ind_{W'}^{W_n}\bigl(\chi^{[\alpha',+]}\otimes \varepsilon_r\bigr)&
(\sigma_{n/2}) -\Ind_{W'}^{W_n}\bigl(\chi^{[\alpha',-]}\otimes \varepsilon_r
\bigr)(\sigma_{n/2})\\
&=(-1)^{n/2} \, 2^{n/2}\,[\SG_{n/2}:\SG_{(n-r)/2} \times \SG_{r/2}]\, 
\chi^{\alpha'}(1).
\end{align*}
On the other hand, Lemma~\ref{solveD1} applied to $W_n$ yields that 
\[\sum_{\alpha} \bigl(\chi^{[\alpha,\mu_\alpha]}(\sigma_{n/2})-
\chi^{[\alpha,-\mu_\alpha]}(\sigma_{n/2})\bigr)=(-1)^{n/2}\, 2^{n/2}
\sum_{\alpha} \mu_\alpha\, \chi^{\alpha}(1),\]
where the sum runs over all partitions $\alpha \vdash n/2$ whose 
Young diagram is obtained from that of $\alpha'$ by adding $r/2$ boxes,
with no two boxes in the same row. Now, by ``Pieri's Rule'' for the 
characters of $\SG_{n/2}$ (see \cite[6.1.7]{gepf}), we have 
\[ [\SG_{n/2}:\SG_{(n-r)/2} \times \SG_{r/2}]\, \chi^{\alpha'}(1)=
\sum_\alpha \chi^\alpha(1)\]
where the sum runs over all $\alpha$ as above. Hence, we conclude that
\begin{multline*}
(-1)^{n/2}\Bigl(\Ind_{W'}^{W_n}\bigl(\chi^{[\alpha',+]}\bigr)(\sigma_{n/2})
-\Ind_{W'}^{W_n}\bigl(\chi^{[\alpha',-]}\bigr)(\sigma_{n/2})\Bigr)=
2^{n/2} \Bigl(\sum_{\alpha} \chi^{\alpha}(1)\Bigr) \\ \geq 2^{n/2} 
\Bigl(\sum_{\alpha} \mu_\alpha \chi^{\alpha}(1)\Bigr)=
(-1)^{n/2}\Bigl(\sum_{\alpha} \bigl(\chi^{\alpha,\mu_\alpha} (\sigma_{n/2})-
\chi^{[\alpha,-\mu_\alpha]}(\sigma_{n/2})\bigr)\Bigr).
\end{multline*}
Since the left hand side equals the right hand side, the inequality
must be an equality which means that $\mu_\alpha=1$ for all $\alpha
\vdash n/2$, as desired.
\end{proof}

Taking the special case $r=2$, we obtain the following strengthened 
version of the ``branching rule'' for type $D_n$.

\begin{cor} \label{solveD2a} Assume that $n \geq 2$ is even. Consider the
parabolic subgroup $W'=W_{n-2}\times H_2$ where $W_{n-2}=\langle
u,s_1,\ldots,s_{n-3}\rangle$ (type $D_{n-2}$) and $H_2=\langle s_{n-1}
\rangle$. Let $\alpha' \vdash (n-2)/2$ and denote by $\varepsilon_1$
the sign character on the factor $H_2$. Then
\[ \Ind_{W'}^{W_n}\bigl(\chi^{[\alpha',+]} \boxtimes \varepsilon_1\bigr)=
\sum_{\alpha} \chi^{[\alpha,+]} \quad+\quad  \mbox{``further terms''},\]
where the sum runs over all partitions $\alpha \vdash n/2$ such that
$\alpha$ is obtained by increasing one part of $\alpha'$ by~$1$; the
expression ``further terms'' stands for a sum of various $\chi \in
\Irr(W_n)$ which can be extended to $W_n$. In particular,
\[ \Big\langle \Ind_{W'}^{W_n}\bigl(\chi^{[\alpha',+]}\boxtimes
\varepsilon_1\bigr),\chi^{[\alpha,-]} \Big\rangle_{W_n}=0 \qquad 
\mbox{for all $\alpha \vdash n/2$}.\]
\end{cor}

Finally, we are able to describe the decomposition of Kottwitz' characters
$\kott_w^1$ and $\kott_w^\diamond$ for $W_n$ into irreducible 
characters. 

\begin{abschnitt} \label{kottspecial} Assume that $n \geq 2$ is even. 
Consider the element $\sigma_{n/2} \in W_n$ in Remark~\ref{speciale}.
Let $\kott_{\sigma_{n/2}}^1$ be the character of the split version of 
Kottwitz' involution module for $W_n$; see Remark~\ref{koco1}. By 
\cite[\S 3.3]{kottwitz}, we have 
\[\kott_{\sigma_{n/2}}^1=\sum_{\alpha \vdash n/2} \chi^{[\alpha,\nu_\alpha]} 
\qquad \mbox{and} \qquad \kott_{t\sigma_{n/2}t}^1=\sum_{\alpha 
\vdash n/2} \chi^{[\alpha,-\nu_\alpha]}\]
where $\nu_\alpha\in \{\pm 1\}$ for all $\alpha \vdash n/2$. But note that 
these signs have not been determined in \cite{kottwitz}. Using an inductive 
argument based on Corollary~\ref{solveD2a}, it is shown in 
\cite[Prop.~7.4]{boge} that  $\nu_\alpha=1$ for all $\alpha \vdash n/2$.
Thus, we have 
\[\kott_{\sigma_{n/2}}^1=\sum_{\alpha \vdash n/2} \chi^{[\alpha,+]} 
\qquad \mbox{and} \qquad \kott_{t\sigma_{n/2}t}^1=\sum_{\alpha \vdash n/2} 
\chi^{[\alpha,-]}.\]
\end{abschnitt}

\begin{abschnitt} \label{clbn} A complete set of representatives of the
conjugacy classes of involutions in $\tW_n$ is given as follows. Let
$l,j$ be non-negative integers such that $l+2j\leq n$. Then set
\[ \sigma_{l,j}:=t_1\cdots t_ls_{l+1}s_{l+3} \cdots s_{l+2j-1}\in \tW_n,\]
where $t_1:=t$ and $t_i:=s_{i-1}t_{i-1}s_{i-1}$ for $2 \leq i \leq n$.
Note that $\sigma_{l,j}$ is the longest element in a parabolic subgroup
of $\tW_n$ of type $B_l \times A_1 \times \ldots \times A_1$, where the
$A_1$ factor is repeated $j$ times. In particular, $\sigma_{l,j}$ is
central in this parabolic subgroup and $\sigma_{l,j}$ has minimal length 
in its conjugacy class; see also \cite[3.2.10]{gepf}. Every involution
in $\tW_n$ is conjugate to exactly one of the elements $\sigma_{l,j}$. 
Note that $\sigma_{l,j} \in W_n$ if and only if $l$ is even. Furthermore,
if $n$ is odd, then every involution in $W_n$ is conjugate (in $W_n$) to 
exactly one of the elements $\sigma_{l,j}$ where $l$ is even. Assume now 
that $n$ is even. Then 
\[ \sigma_{0,n/2}=s_1s_3 \cdots s_{n-1}\in W_n\]
is the element already introduced in Remark~\ref{speciale}.
Let $C_0$ be the conjugacy class of $\sigma_{0,n/2}$ in $W_n$. Recall
from Remark~\ref{speciale} that $C_0^\diamond \neq C_0$ and that $\{C_0,
C_0^\diamond\}$ is the only pair of conjugacy classes of involutions 
in $W_n$ with this property. 
\end{abschnitt}

\begin{abschnitt}  \label{symD} There is an alternative labelling of
$\Irr(\tW_n)$ in terms of Lusztig's ``symbols''. To describe this in more 
detail, let $(\alpha,\beta) \vdash n$ and consider the corresponding
irreducible character $\tchi=\tchi^{(\alpha, \beta)} \in \Irr(\tW_n)$;
see Example~\ref{typeB}. Choose $m \geq 1$ such that we can write
\[ \alpha=(0 \leq \alpha_1 \leq \alpha_2 \leq \ldots \leq \alpha_m)
\qquad \mbox{and} \qquad \beta=(0 \leq \beta_1 \leq \beta_2 \leq \ldots 
\leq \beta_m).\]
As in \cite[\S 1]{LuDn}, \cite[\S 4.18]{LuB}, we have a corresponding 
``symbol'' with two rows of equal length
\[ \Lambda_m(\tchi):=\binom{\lambda_1, \lambda_2,\ldots,\lambda_{m}}{\mu_1,
\mu_2,\ldots,\mu_m} \]
where $\lambda_i:=\alpha+i-1$ and $\mu_i=\beta_i+i-1$ for $1\leq i \leq m$. 
We associate with $\tchi$ the following invariants. First, we set 
\begin{align*}
c(\tchi)& = c(\alpha,\beta) :=\mbox{ number of $i \in \{1,\ldots,m\}$ such 
that $\mu_i \not\in \{\lambda_1,\lambda_2,\ldots,\lambda_m\}$},\\
d_0(\tchi) &=d_0(\alpha,\beta):=\beta_1+\sum_{2 \leq i \leq m} 
\sup\{\alpha_{i-1},\beta_i\}.
\end{align*}
In particular, if $\alpha=\beta$, then $c(\alpha,\alpha)=0$ and $d_0(\alpha,
\alpha)=n/2$. Next, we set 
\begin{align*}
\ab^\diamond_\tchi=\ab_{(\alpha,\beta)}^\diamond&:=\sum_{1\leq i \leq j
\leq m}\min\{\lambda_i,\lambda_j\} +\sum_{1\leq i \leq j \leq m} 
\min\{\mu_i,\mu_j\} \\ &\qquad +\sum_{1\leq i,j \leq m} 
\min\{\lambda_i,\mu_j\}- \frac{1}{6}m(m-1)(4m-5);
\end{align*}
see \cite[4.6.3]{LuB}, \cite[22.14]{lusztig}. (Note that these definitions 
do not depend on the choice of $m$.) 
\begin{itemize}
\item[(a)] We say that $\tchi$ is ``$\diamond$-special'' if $\mu_1 \leq 
\lambda_1 \leq \mu_2 \leq \lambda_2 \leq \ldots \leq \mu_m \leq \lambda_m$.
\end{itemize}
In particular, if $\tchi$ is $\diamond$-special then either $\alpha=\beta$
or $|\beta|<|\alpha|$. 
Assume now that $\alpha \neq \beta$. Then $\chi=\chi^{[\alpha, \beta]}
\in \Irr(W_n)$ and the two extensions of $\chi$ to $\tW_n$ are 
$\tchi^{(\alpha,\beta)}$ and $\tchi^{(\beta,\alpha)}$. Following Lusztig 
\cite[17.2]{cs4}, we say that $\chi^{(\alpha,\beta)}$ is the 
``{\em preferred extension}'' of $\chi$ if the symbol $\Lambda_m(\tchi)$ 
has the following property: the smallest entry which appears in only one 
row appears in the lower row. Using also \cite[4.6.4]{LuB}, one easily 
verifies that:
\begin{itemize}
\item[(b)] $\tchi$ is $\diamond$-special if and only if 
$\ab_{\tchi}^\diamond=\bb_\chi$ and $\tchi$ is the preferred 
extension of $\chi$.
\end{itemize}
This property played a role in the proof of the main result of 
\cite{gema} for type $D_n$, and it will also play a role in the proof
of Proposition~\ref{prefspec} below. 
\end{abschnitt}

\begin{prop}[Kottwitz \protect{\cite[\S 3.3 and \S 5.4]{kottwitz}}]
\label{kottD1} Let $l,j$ be non-negative integers such that $l+2j\leq n$; 
consider the coresponding involution $\sigma_{l,j} \in \tW_n$. Thus, if
$l$ is even, then $\sigma_{l,j} \in W_n$ in an ``ordinary'' involution; 
on the other hand, if $l$ is odd, then $t\sigma_{l,j}\in W_n$ is
a $\diamond$-twisted involution. Let 
\[ {\kott}_{l,j}':=\left\{\begin{array}{cl} \kott_{\sigma_{l,j}}^1 &
\qquad \mbox{if $l$ is even},\\ \kott_{t\sigma_{l,j}}^\diamond & \qquad
\mbox{if $l$ is odd},\end{array}\right.\]
Then the following hold, where $(\alpha,\beta)$ is any pair of partitions
such that $|\alpha|+|\beta|=n$ and, as before, $\chi^{[\alpha,\beta]}$ 
denotes the restriction of $\tchi^{(\alpha,\beta)}\in \Irr(\tW_n)$ to $W_n$.
\begin{itemize}
\item[(a)] We have $\langle {\kott}_{l,j}',\chi^{[\alpha,\beta]}
\rangle_{W_n}=0$ unless $\tchi^{(\alpha,\beta)}$ is $\diamond$-special
and $|\beta|=j$.
\item[(b)] If $\tchi^{(\alpha,\beta)}$ is $\diamond$-special and $|\beta|
=j$, then
\[ \langle {\kott}_{l,j}',\chi^{[\alpha,\beta]}\rangle_{W_n}= 
\binom{c(\alpha,\beta)}{j+l-d_0(\alpha,\beta)} \qquad \mbox{(binomial 
coefficient)};\] 
in particular, the multiplicity is zero unless $d_0(\alpha,\beta)\leq 
j+l \leq d_0(\alpha,\beta) +c(\alpha,\beta)$.
\item[(c)] If $j<n/2$, then $\langle \kott_{l,j}',\chi^{[\alpha,\beta]}
\rangle_{W_n}=0$ unless $\alpha\neq \beta$; if $2j=n$, then $\langle 
\kott_{l,j}',\chi^{[\alpha,\beta]} \rangle_{W_n}=0$ unless $\alpha=
\beta$. 
\end{itemize}
\end{prop}

\begin{proof} First assume that $j<n/2$; then $\sigma_{l,j}$ is
not the special element $\sigma_{n/2}$ in Remark~\ref{speciale}. 
The desired multiplicities in (a) and (b) are explicitly determined in 
\cite[\S 3.3 and \S 5.4]{kottwitz}. A similar argument applies to the
case $j=n/2$ and $l=0$. Using Example~\ref{typeD},
we obtain
\[ \langle {\kott}_{l,j}',\chi^{[\alpha,\alpha]}\rangle_{W_n}= 
\langle \kott_{\sigma_{n/2}}^1,\chi^{[\alpha,+]}+\chi^{[\alpha,-]}
\rangle_{W_n}=1\]
in this case, as already mentioned in (\ref{kottspecial}).
\end{proof}

In order to prove Kottwitz' conjecture, our task now is to find similar
formulae for the numbers of elements in the intersections of left
cells with ordinary or $\diamond$-conjugacy classes of involutions 
in $W_n$.

\begin{rem} \label{dperm1} Let $\cc$ be a two-sided cell of $W_n$
such that $\cc^\diamond\neq \cc$. By \cite[4.6.10, 5.25]{LuB}, this
can only happen if $n$ is even; in this case, we have
\[\Irr(W_n\mid \cc)=\{\chi\} \qquad \mbox{where} \qquad \chi=
\chi^{[\alpha,\pm]} \mbox{ for some $\alpha \vdash n/2$}.\]
Note that this implies that $[\Gamma]_1=\chi$ for every left 
cell $\Gamma\subseteq \cc$. (More precisely, the above statement
on $\Irr(W_n\mid \cc)$ implies that $[\Gamma]_1$ is a multiple of $\chi$;
but then \cite[12.17]{LuB} shows that $[\Gamma]_1$ is multiplicity-free.)
Now let $\Cd$ be a $\diamond$-conjugacy class of $\diamond$-twisted 
involutions in $W_n$.  In Remark~\ref{dperm}, we have seen that 
$\Cd\cap \Gamma=\varnothing$. We can now also show that 
\[ \langle \kott_w^\diamond, [\Gamma]_1\rangle_{W_n}=0 \qquad \mbox{for 
$w \in \Cd$}.\]
Indeed, we are in the case where $l\neq 1$ is odd in Proposition~\ref{kottD1}.
Then the formulae show that all irreducible constituents of 
$\kott_w^\diamond$ are of the form $\chi^{[\alpha,\beta]}$ where
$\alpha \neq \beta$. Hence, we must have $\langle \kott_w^\diamond, 
[\Gamma]_1\rangle_{W_n}=0$ since $[\Gamma]_1=\chi^{[\alpha,\pm]}$ for
some $\alpha \vdash n/2$.
\end{rem}

\section{The extended Iwahori--Hecke algebra} \label{sec:basic}

The aim of this section is to establish certain positivity results for
leading coefficients of character values of Iwahori--Hecke algebras in 
the quasi-split case. The analogous results in the split case where shown 
by Lusztig \cite[7.1]{LuB}, \cite[3.14]{lulc}. The arguments are similar
in the quasi-split case but some additional care is needed in choosing
the correct extensions of the characters in $\Irr^\diamond(W)$. Then
Proposition~\ref{invDn} formulates the main applications to type $D_n$.
We shall need a number of results from Lusztig's book \cite{LuB} and 
\cite{lulc} so, as in Remark~\ref{exthec0}, we assume that $W$ is a Weyl 
group and that $\diamond$ is ``ordinary'' in the sense of \cite[3.1]{LuB}. 
Let $\tilde{W}=\langle W,\gamma\rangle$ be the semidirect product as in
Remark~\ref{semidir}. 

\begin{abschnitt} \label{ainv} We begin by recalling some results 
concerning the representation theory of the (split semisimple) algebra 
$\cH_K=K \otimes_A \cH$, where $K=\Q(v)$ is the field of fractions of $A$. 
(See (\ref{defkl}) for the definition of $\cH$.)
Via the specialisation $v \mapsto 1$, we obtain a canonical bijection
\[ \Irr(W) \leftrightarrow \Irr(\cH_K), \qquad \chi \leftrightarrow \chi_v;\]
see \cite[3.3]{LuB}. We have $\chi_v(T_w) \in A$ for all $w \in W$.
For $\chi \in \Irr(W)$, we define 
\[ \ab_\chi:=\min\{i \geq 0 \mid v^i \chi_v(T_w) \in \Z[v] \mbox{ for all
$w \in W$}\}.\]
Then there are well-defined integers $c_{w,\chi} \in \Z$ such that
\[ v^{\ab_\chi}\chi_v(T_w) \equiv (-1)^{\ell(w)} c_{w,\chi} \, \bmod
v \Z[v] \qquad \mbox{for all $w \in W$}.\] 
These are Lusztig's ``leading coefficients of character values''; see 
\cite[Chap.~5]{LuB}, \cite{lulc}. Note that the sum of all terms $c_{w,
\chi}^2$ ($w \in W$) is a strictly positive number. Consequently, there 
is a well-defined positive rational number $f_{\chi}$ such that 
\[ \sum_{w \in W} c_{w,\chi}^2=\chi(1)\, f_\chi.\]
In fact, it turns out that $f_{\chi}>0$ is an integer; see \cite[4.1]{LuB}. 
We have the following relation; see \cite[Cor.~5.8]{LuB}: 
\[ \sum_{w \in \Gamma} c_{w,\chi}^2=f_\chi\, \langle [\Gamma]_1,
\chi\rangle_W \qquad \mbox{for any left cell $\Gamma$ of $W$}.\]
In particular, if $\cc$ is a two-sided cell of $W$ and $c_{w,\chi} 
\neq 0$ for some $w \in \cc$, then $\chi \in \Irr(W\mid \cc)$. 
We will now have to consider ``quasi-split'' versions of these constructions.
\end{abschnitt}

\begin{abschnitt} \label{exthec1} 
As in \cite[3.3]{LuB}, we can define an extended algebra $\tH$ with 
an $A$-basis $\{T_{\sigma} \mid \sigma \in \tW\}$. For this 
purpose, we define a function $L \colon \tW_n \rightarrow \Z$ by
$L(\gamma^i w)=\ell(w)$ for any $w \in W_n$ and $i=0,1$; in particular,
$L$ is an extension of the length function $\ell$ on $W$. Note that
$L(\gamma)=0$ and $L(\gamma w)=L(w^{-1} \gamma)$ for all $w \in W$. Then 
the multiplication in $\tH$ is given as follows: 
\begin{align*}
T_{\sigma}T_{\sigma'}'&= T_{\sigma \sigma'} \quad \mbox{if $\sigma,
\sigma'\in \tW$ are such that $L(\sigma \sigma')=L(\sigma)+
L(\sigma')$},\\ T_s^2 &= T_1+(v-v^{-1})T_s \quad 
\mbox{if $s \in S$}. \end{align*}
Thus, $\cH$ can be identified with the $A$-submodule of $\tH$ spanned
by all $T_w$ ($w \in W$); note also that $\tH=\cH$ if $\diamond=1$. Let 
$\{\bC_w \mid w \in W\}$ be the Kazhdan--Lusztig basis of $\cH$ as defined
in (\ref{defkl}). Following \cite[3.1(a)]{lulc}, we extend this to a basis 
of $\tH$ by setting 
\[ \bC_{\gamma w}=T_{\gamma w}+\sum_{y \in W, y<w} (-1)^{\ell(w)-\ell(y)}
\,v^{\ell(w)-\ell(y)} \,P_{y,w} T_{\gamma y} \qquad (w \in W).\]
Correspondingly, we also have notions of left, right and two-sided cells
in $\tW$; in order to avoid any confusion with the analogous notions for $W$ 
itself, we shall call them the left, right and two-sided $L$-cells 
in $\tW$. One easily sees that we have the following relations between 
the cells in $W$ and the $L$-cells in $\tW$. 
\begin{itemize}
\item[(a)] If $\Gamma$ is a left cell in $W$, then $\Gamma^+:=
\Gamma \cup \gamma \Gamma$ is a left $L$-cell of $\tW$. All left 
$L$-cells of $\tW$ arise in this way.
\item[(b)] If $\Gamma, \Gamma^+$ are as in (a), then the
corresponding characters of $W, \tW$ are related by
\[ [\Gamma^+]_1=\Ind_{W}^{\tW}\bigl([\Gamma]_1\bigr).\]
\item[(c)] If $\cc$ is a two-sided cell of $W$, then $\cc^+:=\cc \cup 
\gamma \cc \cup \cc \gamma \cup \gamma \cc \gamma$ is a two-sided 
$L$-cell of $\tW$. All two-sided $L$-cells of $\tW$ arise 
in this way.
\end{itemize}
(See \cite[\S 16]{cs3}, \cite[3.1]{lulc}; a related setting is 
considered in \cite[2.4.9]{geja}.)
\end{abschnitt}

\begin{abschnitt} \label{exthec2} We extend the constructions in 
(\ref{ainv}) to $\tH$. Let $K=\Q(v)$ be the field of fractions
of $A$.  Then $\tH_K:=K \otimes_A \tH$ is a split semisimple algebra.
Via the specialisation $v \mapsto 1$, we obtain a canonical bijection
\[\Irr(\tW) \leftrightarrow \Irr(\tH_K),\qquad \tchi \leftrightarrow
\tchi_v\]
which is compatible with the restriction of characters from $\tW$ to $W$
on the one side, and with the restriction of characters from $\tH$ to
$\cH$ on the other side; see \cite[3.3]{LuB}. We have $\tchi_v
(T_\sigma) \in A$ for all $\sigma \in \tW$. For $\tchi \in \Irr(\tW)$, we
define 
\[ \ab_\tchi:=\min\{i \geq 0 \mid v^i \chi_v(T_\sigma) \in 
\Z[v] \mbox{ for all $\sigma \in \tW$}\}.\]
Then there are well-defined integers $c_{\sigma,\tchi} \in \Z$ such that
\[ v^{\ab_\tchi}\tchi_v(T_\sigma) \equiv (-1)^{L(\sigma)} 
c_{\sigma,\tchi} \, \bmod v \Z[v] \qquad \mbox{for all $\sigma \in \tW$}.\] 
Again the sum of all terms $c_{\sigma,\tchi}^2$ ($\sigma \in \tW$) is a 
strictly positive number. Consequently, there is a well-defined positive 
rational number $f_{\tchi}$ such that 
\[ \sum_{\sigma \in \tW} c_{\sigma,\tchi}^2=\tchi(1)\, f_\tchi.\]
The relation between $\ab_{\tchi}$, $f_\tchi$ and the analogous invariants 
for the ireducible characters of $W$ are given as follows; see 
\cite[Prop.~2.4.14]{geja}. Let $\tchi \in \Irr(\tW)$ and $\chi \in 
\Irr(W)$ be such that $\chi$ occurs in the restriction of $\tchi$ to $W$.
Then 
\begin{equation*}
\ab_{\tchi}=\ab_\chi \qquad \mbox{and}\qquad \tchi(1)\,
f_{\tchi} =2\,\chi(1)\,f_\chi.\tag{a}
\end{equation*}
Assume now that $\tchi$ is an extension of some  $\chi=\chi^\diamond
\in \Irr(W)$. Then we have the following relation with the
left cells of $W$; see \cite[Cor.~5.8]{LuB}: 
\begin{equation*}
\sum_{w \in \Gamma} c_{\gamma w,\tchi}^2=f_\chi\, \langle [\Gamma]_1,
\chi\rangle_W \qquad \mbox{for any left cell $\Gamma$ of $W$}.\tag{b}
\end{equation*}
In particular, if $\cc$ is a two-sided cell of $W$ and $c_{\gamma w,\tchi}
\neq 0$ for some $w \in W$, then $\chi \in \Irr(W\mid \cc)$.
\end{abschnitt}

\begin{rem} \label{rem31} Let $\chi \in \Irr(W)$ and assume that
$\chi^\diamond\neq \chi$. Then $\tchi:=\Ind_W^{\tW}(\chi) \in \Irr(\tW)$. 
Correspondingly, if $V$ is an $\cH_K$-module affording $\chi_v$, then 
\[ \tilde{V}=\tH_K \otimes_{\cH_K} V\] 
is a $\tH_K$-module affording $\tchi_v$; note that $\tH_K$ is free
as an $\cH_K$-module, with basis $\{T_1,T_\gamma\}$. One easily sees that 
this implies that 
\[ \tchi_v(T_w)=\chi_v(T_w)+\chi_v^\diamond(T_w) \quad \mbox{and}
\quad \tchi_v(T_{\gamma w})=0 \quad \mbox{for all $w \in W$}.\] 
Hence, we obtain that
\[ c_{w,\tchi}=c_{w,\chi}+c_{w,\chi^\diamond} \quad \mbox{and}
\quad c_{\gamma w,\tchi}=0 \quad \mbox{for all $w \in W$}.\] 
\end{rem}

\begin{abschnitt} \label{dspecial} Recall from \cite[4.1]{LuB} that 
$\chi \in \Irr(W)$ is called ``{\em special}'' if $\ab_\chi=\bb_\chi$. 
By \cite[4.14.2, 5.25]{LuB}, the sets $\Irr(W\mid \cc)$ (where $\cc\subseteq
W$ is a two-sided cell) are explicitly known; in particular, it it is known 
that each set $\Irr(W\mid \cc)$ contains a unique special character. 

Now let $\tilde{\cc}$ be a two-sided $L$-cell of $\tW$ and define
\[ \Irr(\tW\mid \tilde{\cc}):=\big\{\tchi \in \Irr(\tW) \mid \langle
\Ind_W^{\tW}(\chi),\tchi\rangle_{\tW} \neq 0 \mbox{ for some
$\chi \in \Irr(W\mid \cc)$}\big\}\]
where $\cc$ is a two-sided cell of $W$ such that $\tilde{\cc}=\cc^+$.
(One easily sees that this does not depend on the choice of $\cc$;
note that there only is a choice if $\cc\neq \cc^\diamond$.)
Using (\ref{exthec1}), we see that we obtain a partition 
\[ \Irr(\tW)=\coprod_{\tilde{\cc}} \Irr(\tW\mid \tilde{\cc})\]
where $\tilde{\cc}$ runs over all two-sided $L$-cells in $\tW$. 
We now define what it means for an irreducible character $\tchi \in 
\Irr(\tW)$ to be ``{\em $\diamond$-special}''. 
\begin{itemize}
\item[(a)] Assume that the restriction of $\tchi$ to $W$ is not irreducible.
Then we say that $\tchi$ is $\diamond$-special if $\tchi=\Ind_W^{\tW}(\chi)$ 
for some special $\chi \in \Irr(W)$. (Note that $\chi$ is special if and
only if $\chi^\diamond$ is special.)
\item[(b)] Otherwise, $\tchi$ is the extension of some $\chi \in \Irr(W)$.
In this case, we say that $\tchi$ is $\diamond$-special if $\chi$ is special
and $\tchi$ is the ``{\em preferred extension}'' of $\chi$ in the sense of
Lusztig \cite[17.2]{cs4}.
\end{itemize}
With the above definitions, each set $\Irr(\tW \mid \tilde{\cc})$ will also 
contain a unique $\diamond$-special character of $\tW$. Note also that, if 
$W$ of type $D_n$ and $\diamond$ is as in Example~\ref{typeD}, then these
definitions are consistent with those in (\ref{symD}).
\end{abschnitt}

\begin{exmp} \label{invD0} Let $W=W_n$ be of type $D_n$. Let $\cc$ be a 
two-sided cell of $W_n$ and $\chi_0 \in \Irr(W_n\mid \cc)$ be the unique 
special character. Let $\Gamma$ be a left cell contained 
in $\cc$. Then the following hold:
\begin{itemize}
\item[(a)] $[\Gamma]_1$ is multiplicity-free with exactly $f_{\chi_0}$
irreducible constituents (one of which is $\chi_0$); furthermore, $\Gamma$ 
contains exactly $f_{\chi_0}$ involutions of $W_n$. (See \cite[12.17]{LuB}.)
\end{itemize}
Now let $\diamond$ be the non-trivial graph automorphism as in 
Example~\ref{typeD}. We identify $\tilde{W}$ with a group $\tW_n$ of 
type $B_n$. First we note:
\begin{itemize}
\item[(b)] $L\colon \tW_n \rightarrow \Z$ is a weight function in the 
sense of Lusztig \cite{lusztig}. Thus, $\tH$ is the generic 
Iwahori--Hecke algebra associated with $\tW_n,L$ as in \cite{lusztig};
furthermore, the notions of left, right and two-sided $L$-cells of $\tW_n$
in (\ref{exthec1}) correspond exactly to the analogous notions in 
\cite{lusztig}. 

\item[(c)] We have $\ab_{\tchi}=\ab_{\tchi}^\diamond$ for all $\tchi \in 
\Irr(\tW_n)$. (See \cite[22.14]{lusztig}.)
\end{itemize}
Let $\tilde{\cc}$ be a two-sided $L$-cell of $\tW_n$ and $\tchi_0 
\in \Irr(\tW_n\mid \tilde{\cc})$ be the unique $\diamond$-special character.
Let $\tilde{\Gamma}$ be a left $L$-cell contained in $\tilde{\cc}$. 
Then we have, where $c(\tchi_0)$ is defined in (\ref{symD}):
\begin{itemize}
\item[(d)] $[\tilde{\Gamma}]_1$ is multiplicity-free with exactly 
$f_{\tchi_0}=2^{c(\tchi_0)}$ irreducible constituents (one of which is
$\tchi_0$); furthermore, $\tilde{\Gamma}$ contains exactly $f_{\tchi_0}$
involutions of $\tW_n$.
\end{itemize}
Indeed, there are two cases. Assume first that $\tchi_0$ is an extension 
of a special character $\chi_0 \in \Irr(W_n)$. Let $\cc$ be the left cell 
of $W_n$ such that $\chi_0 \in \Irr(W_n\mid \cc)$. By  (\ref{exthec0})(c),
we have $\Irr(W_n \mid \cc) \subseteq \Irr^\diamond(W_n)$. So (a) and 
(\ref{exthec1})(b) imply that $[\tilde{\Gamma}]_1$ is multiplicity-free 
with exactly $2f_{\chi_0}$ irreducible constituents (one of which is
$\tchi_0$). By (\ref{exthec2})(a), we have $f_{\tchi_0}=2f_{\chi_0}$.
Now assume that $\tchi_0$ is obtained by inducing a special 
$\chi_0 \in \Irr(W_n)$ to $\tW_n$. Then $n$ is even and $\chi_0=
\chi^{[\alpha,\pm]}$ for some $\alpha \vdash n/2$. Again, let $\cc$ be 
the left cell of $W_n$ such that $\chi_0 \in \Irr(W_n\mid \cc)$. Then
$[\Gamma]_1=\chi_0$ for any left cell $\Gamma \subseteq \cc$; see 
Remark~\ref{dperm1}. So (\ref{exthec1})(b) implies that $[\tilde{\Gamma}]_1=
\tchi_0=\tchi^{(\alpha,\alpha)}$ is irreducible. By (\ref{exthec2})(b), 
we have $f_{\tchi_0}=f_{\chi_0}$. In both cases, the equality $f_{\tchi_0}=
2^{c(\tchi_0)}$ follows from the explicit formula in \cite[22.14]{lusztig};
note that $c(\tchi_0)=0$ in the second case. This completes the proof of
the statement concerning the decomposition of $[\tilde{\Gamma}]_1$. To prove
the statement concerning the involutions in $\tilde{\Gamma}$, we argue as 
follows. By (b), we can apply \cite[Theorem~1.1]{myfrob} to $\tH$ which 
shows that the number of (ordinary) involutions in $\tilde{\Gamma}$ equals 
the number of irreducible constituents of $[\tilde{\Gamma}]_1$ (counting 
multiplicities). Since $[\tilde{\Gamma}]_1$ is multiplicity-free, this 
yields the desired statement.
\end{exmp}

\begin{abschnitt} \label{exthec3} Assume that $(W,S,\diamond)$ arises from 
a connected reductive algebraic group $\bG$ and a Frobenius map $F$, 
corresponding to some $\F_q$-rational structure on $\bG$. Thus, $W$ is the 
Weyl group of $\bG$ with respect to an $F$-stable maximal torus which is 
contained in an $F$-stable Borel subgroup of $\bG$; furthermore, $w \mapsto 
w^\diamond$ is the map induced by $F$ on $W$. Let $G=\bG^F$. For each 
$\chi \in \Irr^\diamond(W)$, let $R_{\tchi}$ be the corresponding 
``almost character'' of $G$ (see \cite[3.7]{LuB}), where $\tchi\in 
\Irr(\tW)$ is an extension of $\chi$. (If we choose another extension 
$\tchi'$ of $\chi$, then $R_{\tchi'}=\pm R_{\tchi}$.) Let 
\[ \Uch(G)=\{ \rho \in \Irr(G) \mid \langle R_{\tchi},\rho \rangle_G \neq 0
\mbox{ for some $\chi \in \Irr^\diamond(W)$}\}\]
be the set of {\em unipotent characters} of $G$. For any two-sided cell $\cc$
of $W$ such that $\cc^\diamond=\cc$, we denote by $\Uch(G\mid \cc)$ the set 
of all $\rho \in \Uch(G)$ such that $\langle R_{\tchi},\rho\rangle_G \neq 0$ 
for some $\chi \in \Irr(W\mid \cc)$. By the ``Disjointness Theorem'' 
\cite[6.16]{LuB}, we obtain a partition
\[ \Uch(G)=\coprod_{\cc} \Uch(G\mid \cc)\]
where $\cc$ runs over all two-sided cells of $W$ such that $\cc^\diamond=
\cc$. All the multiplicities $\langle R_{\tchi}, \rho\rangle_G$ are 
explicitly described by \cite[Main Theorem~4.23]{LuB}; this involves a 
certain Fourier matrix and a function $\Delta \colon \Uch(G) \rightarrow 
\{\pm 1\}$.
\end{abschnitt}

To state the following result, we introduce the following notation.
For $\chi \in \Irr^\diamond(W)$, we set 
\[ c_{\gamma w,\tchi}:=(-1)^{\ab_\chi+l(w)}\, c_{\gamma w,\tchi}
\qquad \mbox{for all $w \in W$},\]
where $\tchi\in \Irr(\tW)$ is an extension of $\chi$ to $\tW$. 

\begin{prop}[Lusztig \protect{\cite[7.1]{LuB}, \cite{lulc}}] \label{lu71} In 
the above setting, let $\cc$ be a two-sided cell of $W$ such that 
$\cc^\diamond=\cc$. Assume there exists some $\chi_0 \in \Irr^\diamond(W)$ 
and an extension $\tchi_0\in \Irr(\tW)$ of $\chi_0$ such that
\begin{equation*}
\Delta(\rho)\langle R_{\tchi_0},\rho\rangle_G>0 \qquad \mbox{for all 
$\rho \in \Uch(G^F\mid \cc)$}.\tag{$*$}
\end{equation*}
Then $c_{\gamma w,\tchi_0}^*\geq 0$ for all $w \in \cc$; furthermore, 
$c_{\gamma w,\tchi_0}^*>0$ for all $w \in \cc$ such that $w^\diamond, 
w^{-1}$ belong to the same left cell of $W$.
\end{prop}

\begin{proof} First consider the inequality $c_{\gamma w,\tchi_0}^*\geq 0$ 
for all $w \in \cc$. In \cite[7.1]{LuB}, this is proved assuming that $F$ acts 
trivially on $W$ and $\Delta(\rho) =1$ for all $\rho \in \Uch(G\mid \cc)$. 
But the same proof gives the more general statement above. Let us briefly 
sketch the main ingredients. For each $\chi \in \Irr^\diamond(W)$ (other 
than $\chi_0$), let us fix some extension $\tchi\in \Irr(\tW)$. Let 
$w\in \cc$ and consider the class function 
\[ R_{\gamma w}:=\sum_{\chi \in \Irr^\diamond(W)} c_{\gamma w,\tchi}
\, R_{\tchi}\]
on $G$. (Note that this is independent of any choices.) Since $w \in
\cc$, we have that $R_{\gamma w}$ is a linear combination of the 
unipotent characters in $\Uch(G\mid \cc)$; see \cite[5.2]{LuB}. Since 
the functions $\{R_{\tchi}\mid \chi \in \Irr^\diamond(W)\}$ form an 
orthonormal system (see \cite[3.9]{LuB}), we obtain
\[ c_{\gamma w,\tchi_0}=\sum_{\rho \in \Uch(G\mid \cc)} \langle 
R_{\tchi_0}, \rho\rangle_G\, \langle R_{\gamma w}, \rho\rangle_G .\] 
Now let $\rho \in \Uch(G\mid \cc)$ be such that the corresponding terms in
the above sum are non-zero. Then \cite[6.19]{LuB} shows that 
$(-1)^{\ab_{\chi_0}+\ell(w)}=\Delta(\rho)$. Hence, we obtain 
\[ c_{\gamma w,\tchi_0}^*=(-1)^{\ab_{\chi_0}+\ell(w)}c_{\gamma w, \tchi_0}=
\sum_{\rho \in \Uch(G\mid \cc)} \Delta(\rho)\langle R_{\tchi_0}, 
\rho\rangle_G\, \langle R_{\gamma w}, \rho\rangle_G .\] 
Now, by the ``Disjointness Theorem'' \cite[6.17]{LuB}, $R_{\gamma w}$ is 
an actual character of $G$ and so $\langle R_{\gamma w}, \rho\rangle_G
\geq 0$ for all $\rho \in \Uch(G\mid \cc)$. Since ($*$) is assumed
to hold, we conclude that $c_{\gamma w,\tchi_0}\geq 0$ for all $w \in \cc$.

Finally, assume that $w \in \cc$ is such that $w^\diamond, w^{-1}$ belong 
to the same left cell of $W$. Then we must prove that $c_{\gamma w,
\tchi_0}^* \neq 0$. Now, the above expression for $c_{\gamma w, \tchi_0}^*$ 
(together with ($*$)) shows that it will be sufficient to prove that that 
there exists some $\rho \in \Uch(G\mid \cc)$ such that $\langle 
R_{\gamma w}, \rho\rangle_G\neq 0$. For this purpose, it is enough to 
show that $R_{\gamma w} \neq 0$. Furthermore, since the class functions 
$\{R_\tchi\}$ are linearly independent, it will be sufficient to show that
$c_{\gamma w,\tchi} \neq 0$ for some $\chi \in \Irr^\diamond(W)$. But this 
follows by an argument involving Lusztig's asymptotic algebra $\tJ$; see 
\cite[3.1]{lulc}. Indeed, this algebra has a basis $\{t_{\sigma} \mid \sigma
\in \tW\}$ where the structure constants are integers. It is known that
$\tJ$ is a ``based ring'' in the sense of \cite[\S 1]{lulc}; see 
\cite[3.1(j)]{lulc}. This has several consequences. First of all, by 
\cite[3.1(k)]{lulc}, the elements $\sigma, \sigma^{-1}$ belong to the 
same left $L$-cell in $\tW$ if and only if $t_{\sigma}^2\neq 0$. 
Furthermore, by \cite[1.2(b)]{lulc}, we have $t_{\sigma}^2\neq 0$ if and 
only some irreducible character of $\tJ$ has a non-zero value on 
$t_{\sigma}$. Finally, by \cite[3.4(a), (e)]{lulc}, the leading coefficients 
$c_{\sigma,\tilde{\psi}}$ can be interpreted (up to signs) as the 
values of the irreducible characters of $\tJ$ on $t_{\sigma}$.  Thus,
we have:
\[ \sigma,\sigma^{-1} \mbox{ belong to the same left $L$-cell} 
\qquad \Leftrightarrow  \qquad c_{\sigma,\tilde{\psi}}\neq 0 \mbox{ for
some $\tilde{\psi} \in \Irr(\tW)$}.\]
Now return to our element $w \in \cc$ such that $w^\diamond, w^{-1}$ 
belong to the same left cell of $W$. Since $\diamond$ permutes the left 
cells of $W$, we also have that $w,(w^\diamond)^{-1}$ belong to the same 
left cell of $W$. Consequently, since $(\gamma w)^{-1}=w^{-1}\gamma=
\gamma (w^\diamond)^{-1}$, the elements $\gamma w,(\gamma w)^{-1}$ belong 
to the same left $L$-cell of $\tW$. So the above equivalence shows that 
there exists some $\tilde{\psi}\in \Irr(\tW)$ such that $c_{\gamma w,
\tilde{\psi}} \neq 0$. But then Remark~\ref{rem31} implies that 
$\tilde{\psi}$ must be an extension of some $\psi \in \Irr^\diamond(W)$, 
as required.
\end{proof}

\begin{prop} \label{prefspec} Assume that we are in the setting of 
(\ref{exthec3}). Let $\cc$ be a two-sided cell of $W$ such 
that $\cc^\diamond=\cc$. Let $\chi_0\in \Irr(W)$ be the unique special 
character in $\Irr(W\mid\cc)$; we have $\chi_0^\diamond =\chi_0$. Then 
condition $(*)$ in Proposition~\ref{lu71} holds if $\tchi_0$ is the 
$\diamond$-special extension of $\chi_0$ in the sense of (\ref{dspecial}).
\end{prop}

\begin{proof} By standard reduction arguments, it is enough to prove this
in the case where $W$ is irreducible. If $F$ acts trivially on $W$, then 
the multiplicity formula in \cite[Main Theorem 4.23]{LuB} shows that
$\langle R_{\tchi_0},\rho\rangle_G=\Delta(\rho)$ for all $\rho \in
\Uch(G\mid \cc)$. (The special character $\chi_0$ corresponds to the
pair $(1,1)$ in the set $\cM(\cG_\cc)$ where $\cG_\cc$ is the finite
group associated with $\cc$.) Hence, the assertion is clear in this case. 
So let us now assume that $F$ does not act trivially on $W$. Then we only 
have $2$ cases to consider:
\begin{itemize}
\item[(a)] $W$ is of type $E_6$ or $A_n$ and $\diamond$ is given by
conjugation with the longest element.
\item[(b)] $W=W_n$ is of type $D_n$ and $\diamond$ is given as in
Example~\ref{typeD}.
\end{itemize}
Assume that we are in case (a). Then $\chi^\diamond=\chi$ for all
$\chi \in \Irr(W)$. Given $\chi$, the ``preferred extension''
$\tchi$ is determined by the condition that $\gamma w_0$ acts
as $(-1)^{\ab_\chi}$ in a representation affording $\tchi$. The assertion
then follows from the description of the Fourier matrix in 
\cite[4.19]{LuB} and the $\Delta$-function in \cite[p.~124]{LuB}.
Now assume that we are in case (b). The ``preferred extensions''
are described in (\ref{symD}). We have $\Delta(\rho)=1$ for all 
$\rho \in \Uch(G)$; see \cite[6.18.5, 6.19]{LuB}. The assertion now 
follows from the description of the Fourier matrix in \cite[4.18]{LuB}; 
see also \cite[Theorem~3.15]{LuDn}. (This has also been discussed in some 
detail in the proof of \cite[Theorem~5.1]{gema}.)
\end{proof}

\begin{lem}[``The basic identity''; cf.\  \protect{\cite[\S 3]{boge}}] 
\label{bi} Let $\cc$ be a two-sided cell of $W$ such that
$\cc^\diamond=\cc$. Let $\Gamma$ be a left cell of $W$ such that $\Gamma
\subseteq \cc$ and $\Cd$ be a $\diamond$-conjugacy class of 
$\diamond$-twisted involutions of $W$.  Then 
\[ \langle [\Gamma]_1,\chi\rangle_{W} \sum_{w \in \Cd \cap \cc} 
c_{\gamma w,\tchi}=\chi(1) \sum_{w \in \Cd \cap \Gamma} c_{\gamma w,
\tchi} \qquad \mbox{for all $\chi \in \Irr^\diamond(W)$},\]
where $\tchi\in \Irr(\tW)$ denotes a fixed extension of $\chi\in 
\Irr^\diamond(W)$ to $\tW$.
\end{lem}

\begin{proof} If $\diamond=1$, then this is proved in \cite[Lemma~3.1]{boge}.
The general case is completely analogous. First, as in the proof of 
\cite[Lemma~1.2]{boge}, one verifies that $T_{s^\diamond}Z= ZT_s$ for all 
$s \in S$, where 
\[ Z:=\sum_{w \in \Cd} (-1)^{\ell(w)}\,T_w \in \cH.\]
Consequently, the element $\tilde{Z}:=\sum_{w \in \Cd} (-1)^{\ell(w)}\, 
T_{\gamma w} \in \tH$ lies in the centre of $\tH$. Once this is 
established, the proof proceeds exactly as in \cite[Lemma~3.1]{boge}. All 
the required properties of the leading coefficients $c_{\gamma w,\tchi}$ and 
the structure constants of the Kazhdan--Lusztig basis of $\tH$ hold 
by \cite[3.1--3.4]{lulc}.
\end{proof}


We now apply the above results to type $D_n$. 

\begin{prop} \label{invDn} Let $W=W_n$ be of type $D_n$ and $\diamond$ be 
as in Example~\ref{typeD}. We identify $\tilde{W}$ with a group $\tW_n$ 
of type $B_n$. Let $\cc$ be a two-sided cell of $W_n$ and $\chi_0 \in 
\Irr(W_n)$ be the unique special character in $\Irr(W_n\mid \cc)$. Then 
the following hold.
\begin{itemize}
\item[(a)] Assume that $\cc^\diamond=\cc$ and let $\tchi_0 \in \Irr(\tW_n)$ 
be the preferred extension of $\chi_0=\chi_0^\diamond$. Then 
\[  c_{tw,\tchi_0}^*=1 \qquad \mbox{for all $\diamond$-twisted involutions 
$w \in \cc$}.\]
furthermore, each left cell $\Gamma \subseteq \cc$ contains exactly 
$f_{\chi_0}$ $\diamond$-twisted involutions.
\item[(b)] Let $\Cd$ be a $\diamond$-conjugacy class of 
$\diamond$-twisted involutions in $W_n$. Then 
\[ |\Cd\cap \cc|=\chi_0(1)|\Cd \cap \Gamma| \qquad
\mbox{for any left cell $\Gamma \subseteq \cc$}.\]
\end{itemize}
\end{prop}

\begin{proof} (a) Let $w \in \cc$ be a $\diamond$-twisted involution. Then
$w^\diamond=w^{-1}$ belong to the same left cell. By Propositions~\ref{lu71} 
and \ref{prefspec}, we conclude that $c_{tw, \tchi_0}^*>0$. In order to show 
that $c_{tw,\tchi_0}^*=1$, we use a counting argument. Let $\Gamma 
\subseteq \cc$ be any left cell. Since $c_{tw,\tchi_0}^*>0$ for all 
$\diamond$-twisted involutions $w \in \Gamma$, we have 
\[ \mbox{(number of $\diamond$-twisted involutions in $\cc$)} 
\leq \sum_{w \in \cc} c_{tw,\tchi_0}^2, \]
with equality only if $c_{tw,\tchi_0}^*=1$ for all $\diamond$-twisted 
involutions $w\in\Gamma$. By (\ref{exthec2})(b), the right hand side 
equals $f_{\chi_0} \langle [\Gamma]_1,\chi_0 \rangle_{W_n}$. By 
Example~\ref{invD0}(a), we have $\langle [\Gamma]_1,\chi_0\rangle_{W_n}=1$
and so 
\[ \mbox{(number of $\diamond$-twisted involutions in $\Gamma$)}  \leq
f_{\chi_0}.\]
Now consider the left $L$-cell $\Gamma^+=\Gamma \cup t\Gamma$ of
$\tW_n$. Using Example~\ref{invD0}(a) and Remark~\ref{semidir}, the above
inequality can be rephrased as:
\[\mbox{(number of ordinary involutions in $\Gamma^+$)}\leq 2f_{\chi_0},\]
where equality holds if and only if equality holds in all the previous
inequalities. But then Example~\ref{invD0}(d) shows that all the previous 
inequalities must be equalities; note that $f_{\tchi_0}=2f_{\chi_0}$
in this case. In particular, $c_{tw,\tchi_0}^*=1$ for all $\diamond$-twisted 
involutions $w \in \Gamma$. It also follows that the number of 
these $\diamond$-twisted involutions equals~$f_{\chi_0}$. Thus, (a) is proved.

(b) If $\cc^\diamond \neq \cc$, then both sides of the equality are zero
by Remark~\ref{dperm}. So let us now assume that $\cc^\diamond=\cc$ and 
let $\tchi_0$ be as in (a). By Example~\ref{invD0}(a), we have 
$\langle [\Gamma]_1, \chi_0\rangle_{W_n}=1$. It remains to use the 
identity in Lemma~\ref{bi}.
\end{proof}

\section{Twisted involutions in type $D_n$} \label{sec:tDn}

Throughout this section we place ourselves in the setting of 
Example~\ref{typeD}. Thus, $n \geq 2$ and $W=W_n$ is a Coxeter
group of type $D_n$, with generators $u,s_1,\ldots,s_{n-1}$. Let 
$w \mapsto w^\diamond$ be defined by $u^\diamond=s_1$, $s_1^\diamond=u$ 
and $s_i^\diamond=s_i$ for $2 \leq i \leq n-1$. We identify the semidirect 
product $\tW=W \rtimes \langle \diamond\rangle$ with the Coxeter group 
$\tW_n$ of type $B_n$ with generators $t,s_1,\ldots,s_{n-1}$ as in 
Example~\ref{typeB}. Recall that, under this identification, we have
\[ w^\diamond=twt \qquad \mbox{for all $w \in W_n$}.\]
Let $L \colon \tW_n \rightarrow \Z$ be defined as in (\ref{exthec1}).
We have already noted in Example~\ref{invD0}(c) that $L$ is a weight 
function in the sense of Lusztig \cite{lusztig}; explicitly, we have
\[ L(t)=0 \qquad \mbox{and} \qquad L(s_1)=L(s_2)=\ldots =L(s_{n-1})=1.\]
We shall consider the left, right and two-sided $L$-cells of $\tW_n$.
In Theorem~\ref{maint}, we state a modified version of Kottwitz'
conjecture for all (ordinary) conjugacy classes of involutions in $\tW_n$. 
This crucially relies on the construction of the modified 
involution module in Lemma~\ref{typeD2}. Then note that any involution 
in $\tW_n$ is either an ordinary involution in $W_n$ or corresponds to a 
$\diamond$-twisted involution in $W_n$. In Corollary~\ref{mainc}, we 
will see that the modified version of Kottwitz' conjecture for $\tW_n$ 
encapsulates both the split and the quasi-split version of Kottwitz' 
conjecture for $W_n$.

Throughout, it will be convenient to allow also the possibility
that $n=0,1$, where $\tW_0=\{1\}$, $\tW_1=\{1,t\}$ and $W_0=W_1=\{1\}$.

\begin{rem} \label{remBD}
We have the following relation between the length function $\ell$ on 
$W_n$ and the length function $\tell$ on $\tW_n$. Let $w\in W_n$ and 
$1\leq i \leq n-1$. Then we have 
\[\ell(ws_i)<\ell(w) \quad\Longleftrightarrow\quad
\tell(ws_i)<\tell(w) \quad\Longleftrightarrow\quad
\tell(tws_i)<\tell(tw).\]
Indeed, first note that, by \cite[1.4.12]{gepf}, we have $\tell(w)=\ell(w)+
\ell_t(w)$, where $\ell_t(w)$ denotes the number of occurrances of $t$
in a reduced expression of $w$ in terms of the generators of $\tW_n$. This
implies the first equivalence. To prove the second equivalence, we 
distinguish two cases. Suppose first that $\tell(tw)>\tell(w)$. Now, if
$\tell(ws_i)<\tell(w)$, then $\tell(tws_i)\leq \tell(w)<\tell(tw)$, as 
required. Conversely, assume that $\tell(tws_i)< \tell(tw)$. Since 
$\tell(tw)>\tell(w)$, this implies $\tell(tws_i)=\tell(w)$. So we must 
have $\tell(ws_i)<\tell(w)$ by \cite[Lemma~1.2.6]{gepf}. (Otherwise, we 
would have $tws_i=w$ and so $t,s_i$ would be conjugate, a contradiction.) 
The argument for the case $\tell(tw)<\tell(w)$ is similar.
\end{rem}

\begin{lem} \label{typeD2} Let $\CC$ be an (ordinary) conjugacy class of 
involutions of $\tW_n$. Let $\tilde{M}$ be a $\Q$-vector space with a 
basis $\{\tilde{a}_\sigma \mid \sigma\in \CC\}$. Then $\tilde{M}$ is a
$\Q[\tW_n]$-module, where the action is given by: 
\begin{align*}
t.\tilde{a}_\sigma& =\tilde{a}_{t\sigma t},\\
s_i.\tilde{a}_{\sigma}&=\left\{\begin{array}{cl} -\tilde{a}_{\sigma} 
& \qquad \mbox{if $s_i\sigma=\sigma s_i$ and $\tell(\sigma s_i)<
\tell(\sigma)$},\\ \tilde{a}_{s_i\sigma s_i} 
& \qquad \mbox{otherwise},  \end{array}\right.
\end{align*}
for $1\leq i \leq n-1$. Furthermore, let $\tilde{\kott}_{\CC}$ denote the
character of $\tW_n$ afforded by $\tilde{M}$. 
\begin{itemize}
\item[(a)] If $\CC \subseteq W_n$ and $\CC$ is a single conjugacy class
in $W_n$, then the restriction of $\tilde{\kott}_{\CC}$ to $W_n$ is 
Kottwitz' character $\kott_{\CC}^1$ for $W_n$ (split case, see 
Remark~\ref{koco1}).
\item[(b)] If $\CC \subseteq W_n$ and $\CC$ consists of two conjugacy classes
in $W_n$, then $n$ is even and the restriction of $\tilde{\kott}_{\CC}$ to 
$W_n$ equals the sum of the two characters $\kott_{\sigma_{n/2}}^1$ and 
$\kott_{t\sigma_{n/2}t}^1$ in (\ref{kottspecial}). 
\item[(c)] If $\CC \subseteq W_nt$, then $\Cd=\{ t\sigma \mid \sigma \in
\CC\}$ is a $\diamond$-conjugacy class of $\diamond$-twisted involutions
in $W_n$ and the restriction of $\tilde{\kott}_{\CC}$ to $W_n$ is 
Kottwitz' character $\kott_{\Cd}^\diamond$ for $W_n$ (quasi-split case).
\end{itemize}
\end{lem}

\begin{proof} If $\CC\subseteq W_n$, then this result is contained in 
\cite[Prop.~2.4]{gema}, with the only difference that the length condition 
is expressed in terms of the length $\ell$ on $W_n$. But then the first
equivalence in Remark~\ref{remBD} allows us to rewrite this condition as
above. By \cite[Rem.~2.2]{gema}, the character of the restriction of 
$\tilde{M}$ to $W_n$ is $\kott_{\CC}^1$. This yields (a) if $\CC$ is 
a single conjugacy class in $W_n$; otherwise, $n$ must be even and
$\CC=C_0\cup tC_0t$ where $C_0$ is the conjugacy class of the element
$\sigma_{n/2}$ in Remark~\ref{speciale}; this yields (b).

Now assume that $\CC \subseteq tW_n$. Then recall from Remark~\ref{semidir} 
that $\Cd:=\{t\sigma \mid \sigma\in \CC\}$ is a $\diamond$-conjugacy class 
of $\diamond$-twisted involutions in $W_n$. Let $M$ be a $\Q$-vector space 
with a basis $\{a_w \mid w\in \Cd\}$. By \cite[7.1]{LV} (see also
\cite{Lu12a}), we already know that $M$ is a $\Q[W_n]$-module, where the 
action is given by the following formulae for any $s \in \{u,s_1,\ldots,
s_{n-1}\}$:
\[ s.a_w=\left\{\begin{array}{cl} -a_w  & \qquad \mbox{if $s^\diamond w=ws$ 
and $\ell(ws)<\ell(w)$},\\ a_{s^\diamond ws} & \qquad \mbox{otherwise}.
\end{array}\right.\]
By (\ref{luvo}), the character of $W_n$ afforded by $M$ is 
$\kott_{\Cd}^\diamond$. Via the linear map $M\rightarrow \tilde{M}$, 
$a_w \mapsto \tilde{a}_{tw}$, we can transport this action to $\tilde{M}$.
The action of $W_n$ on $\tilde{M}$ is given by the following formulae for 
any $s \in \{u,s_1, \ldots,s_{n-1}\}$:
\[ s.\tilde{a}_{\sigma}=\left\{\begin{array}{cl} -\tilde{a}_{\sigma}  
& \qquad \mbox{if $s\sigma=\sigma s$ and $\ell(t\sigma s)<
\ell(t\sigma)$},\\ \tilde{a}_{s\sigma s} & \qquad \mbox{otherwise}.
\end{array} \right.\]
Now, by \cite[0.4]{Lu12b}, this action can be extended to $\tW_n$ via the
formulae:
\begin{align*}
t.\tilde{a}_{\sigma}& =\tilde{a}_{t\sigma t},\\ 
s_i.\tilde{a}_{\sigma}&=\left\{\begin{array}{cl} -\tilde{a}_{\sigma} 
& \qquad \mbox{if $s_i\sigma=\sigma s_i$ and $\ell(t\sigma s_i)<
\ell(t\sigma)$},\\ \tilde{a}_{s_i\sigma s_i} & \qquad \mbox{otherwise};
\end{array}\right.
\end{align*}
Then Remark~\ref{remBD} shows that the conditions involving the
length function $\ell$ on $W_n$ can be rewritten in terms of the length 
function $\tell$ on $W_n$. 
\end{proof}

We can now state the main result of this section.

\begin{thm} \label{maint} Let $\CC$ be any (ordinary) conjugacy class of 
involutions in $\tW_n$. Then 
\begin{center}
\fbox{$\langle \tilde{\kott}_\CC,[\tilde{\Gamma}]_1\rangle_{\tW_n}=
|\CC \cap \tilde{\Gamma}| \qquad \mbox{for any left $L$-cell 
$\tilde{\Gamma}\subseteq \tW_n$};$}
\end{center}
here, $\tilde{\kott}_\CC$ is the character of the $\Q[\tW_n]$-module
$\tilde{M}$ in Lemma~\ref{typeD2}.
\end{thm}

The proof will be given in (\ref{prooft}), at the end of this section.

\begin{cor} \label{mainc} Assuming the truth of Theorem~\ref{maint}, both 
the split and the quasi-split version of Kottwitz' Conjecture~\ref{koco}
hold for $W_n$.
\end{cor}

\begin{proof} As far as the split version is concerned, the argument
is given in \cite[Cor.~7.6]{boge}; this also uses the formulae in 
(\ref{kottspecial}) for $\kott_{\sigma_{n/2}}^1$ and
$\kott_{t\sigma_{n/2}t}^1$ (where the signs $\nu_\alpha$ have been fixed).
Now consider the quasi-split case. Let $\Cd$ be a $\diamond$-conjugacy class
of $\diamond$-twisted involutions of $W_n$. Then $\CC:=\{tw \mid w \in\Cd\}$
is a conjugacy class of involutions of $\tW_n$ which is contained in
the coset $tW_n$; see Remark~\ref{semidir}. Let $\Gamma$ be a left cell 
of $W_n$ and $\tilde{\Gamma}=\Gamma \cup t\Gamma$ be the corresponding 
left $L$-cell of $\tW_n$. Then, using (\ref{exthec1})(b), 
Lemma~\ref{typeD2}(c), Frobenius reciprocity and Theorem~\ref{maint}, we 
obtain 
\[ \langle \kott_{\Cd}^\diamond,[\Gamma]_1\rangle_{W_n}=\langle 
\tilde{\kott}_{\CC}, [\tilde{\Gamma}]_1\rangle_{\tW_n}=|\CC \cap 
\tilde{\Gamma}|.\]
Since $\CC \subseteq tW_n$, the right hand side equals $|\CC\cap t\Gamma|=
|t(\Cd \cap \Gamma)|=|\Cd \cap \Gamma|$, as required.
\end{proof}

We now turn to the proof of Theorem~\ref{maint}; this will require a
number of preparations. 

\begin{prop} \label{kottD2} Let $\CC$ be a conjugacy class of
involutions in $\tW_n$. Assume that $\sigma_{l,j}\in \CC$ where 
$l,j \geq 0$ are such that $l+2j\leq n$; see (\ref{clbn}). Then, using 
the notation in (\ref{symD}), we have
\[ \tilde{\kott}_{\CC}=\sum_{(\alpha,\beta)} \binom{c(\alpha,
\beta)}{j+l- d_0(\alpha,\beta)}\, \tchi^{(\alpha,\beta)}\]
where the sum runs over all $(\alpha,\beta) \vdash n$ such that 
$\tchi^{(\alpha,\beta)}$ is $\diamond$-special, $|\beta|=j$ and 
$d_0(\alpha,\beta)\leq j+l \leq d_0(\alpha,\beta) +c(\alpha,\beta)$.
\end{prop}

\begin{proof} Assume first that $n$ is even, $l=0$ and $j=n/2$. Then, by 
(\ref{kottspecial}) and Lemma~\ref{typeD2}(b), the restriction of 
$\tilde{\kott}_{\CC}$ to $W_n$ equals
\[ \kott_{\sigma_{n/2}}^1+\kott_{t\sigma_{n/2}t}^1=\sum_{\alpha \vdash n/2} 
\bigl(\chi^{[\alpha,+]}+\chi^{[\alpha,-]}\bigr).\]
So Frobenius reciprocity immediately implies that $\tilde{\kott}_{\CC}=
\sum_{\alpha \vdash n/2} \chi^{(\alpha,\alpha)}$, in accordance with the 
formula stated above. Now assume that $j<n/2$ and let $(\alpha,\beta)$ be 
a pair of partitions such that $|\alpha|+ |\beta|=n$. Then Frobenius 
reciprocity, Proposition~\ref{kottD1} and Lemma~\ref{typeD2} show that 
\[ \big\langle \tilde{\kott}_{\CC},\Ind_{W_n}^{\tW_n}(\chi^{[\alpha,
\beta]}) \big\rangle_{\tW_n}= 0 \qquad \mbox{unless} \qquad  \alpha\neq
\beta,\; |\beta|=j \mbox{ and $\tchi^{(\alpha,\beta)}$ is 
$\diamond$-special}; \]
furthermore, if $\alpha \neq \beta$, $|\beta|=j$ and $\tchi^{(\alpha,
\beta)}$ is $\diamond$-special, then 
\[ \big\langle \tilde{\kott}_{\CC},\Ind_{W_n}^{\tW_n}(\chi^{[\alpha,
\beta]}) \big\rangle_{\tW_n}= \mbox{ binomial coefficient as above}.\]
Now let $\alpha \neq \beta$. Since  
\[ \Ind_{W_n}^{\tW_n}(\chi^{[\alpha,\beta]})=\tchi^{(\alpha,\beta)}+
\tchi^{(\beta,\alpha)},\]
it will be sufficient to show that 
\begin{equation*}
\big\langle \tilde{\kott}_{\CC},\tchi^{(\alpha,\beta)}
\big\rangle_{\tW_n} =0 \qquad \mbox{unless} \qquad |\alpha|>|\beta|.\tag{$*$}
\end{equation*}
Now, by Example~\ref{clbn}, we can assume that $\sigma=\sigma_{l,j}$ is the 
longest element in a parabolic subgroup $\tW_I \subseteq \tW_n$ where $I
\subseteq \{t,s_1,\ldots,s_{n-1}\}$; furthermore, $\sigma$ is central in 
$\tW_I$. Then $C_{\tW_n}(\sigma)=\tilde{Y} \ltimes \tW_I$ where $\tilde{Y}$
is a certain set of distinguished coset representatives of $\tW_I$ in 
$\tW_n$; see \cite[\S 2]{gema}. By the argument in \cite[Lemma~2.1]{gema} 
(see also \cite[Rem.~3.3]{gema}; this is essentially the same argument
as in (\ref{luvo})), we have
\[ \tilde{M} \cong \Ind_{C_{\tW_n}(\sigma)}^{\tW_n}
(\tilde{\varepsilon}_\sigma)\] 
where the homomorphism $\tilde{\varepsilon}_\sigma \colon C_{\tW_n}(\sigma)
\rightarrow \{\pm 1\}$ is given by 
\[ \tilde{\varepsilon}_\sigma(yw')=(-1)^{\tell(w')-\ell_t(w')}
\qquad \mbox{for all $y \in \tilde{Y}$ and $w'\in \tW_I$}.\]
Then ($*$) is shown in \cite[Lemma~3.4]{gema}. (Note that, in \cite{gema}, 
it is generally assumed that $\CC \subseteq W_n$ but in the proof of 
\cite[Lemma~3.4]{gema}, this assumption is irrelevant.)
\end{proof}

\begin{rem} \label{invBd} Let $X\subseteq \tW_n$ be any subset which is a
union of (ordinary) conjugacy classes of involutions in $\tW_n$. Then we set
\[ \tilde{\kott}_X=\sum_\CC \tilde{\kott}_\CC\]
where $\CC$ runs over the conjugacy classes contained in $X$. In particular, 
this applies to the set of all involutions in $\tW_n$, which we denote 
by $\cI_n$. With this notation, we have
\begin{center}
\fbox{$\displaystyle \tilde{\kott}_{\cI_n}=\sum_{\tchi \in \Irr(\tW_n) 
\text{ $\diamond$-special}} 2^{c(\tchi)}\, \tchi$.}
\end{center}
This immediately follows from Proposition~\ref{kottD2}, by summing over
all $l,j\geq 0$ such that $l+2j\leq n$.
\end{rem}

\begin{lem} \label{invDn1} Let $\tilde{\cc}$ be a two-sided $L$-cell 
of $\tW_n$ and $\CC$ be any ordinary conjugacy class of involutions in
$\tW_n$. Then 
\[ |\CC\cap  \tilde{\cc}|=\tchi_0(1)|\CC\cap\tilde{\Gamma}| 
\qquad \mbox{for any left $L$-cell $\tilde{\Gamma} \subseteq 
\tilde{\cc}$},\]
where $\tchi_0 \in \Irr(\tW_n)$ is the unique $\diamond$-special character 
in $\Irr(\tW_n \mid \tilde{\cc})$.  
\end{lem}

\begin{proof} Let $\cc$ be a two-sided cell of $W_n$ such that 
$\tilde{\cc}=\cc^+$; see (\ref{exthec1}). Then we also have
$\tilde{\Gamma}=\Gamma \cup t\Gamma$ for some left cell $\Gamma 
\subseteq \cc$ of $W_n$. Let $\chi_0 \in \Irr(W_n\mid \cc)$
be the unique special character. By \cite[Exp.~4.5]{boge}, we already 
konw that 
\begin{equation*}
|\CC \cap \cc|=\chi_0(1)|\CC \cap \Gamma| \qquad \mbox{if $\CC
\subseteq W_n$}.\tag{$*$}
\end{equation*}
We now distinguish two cases.

{\em Case 1}. Assume that $\cc^\diamond=\cc$. Then $\tilde{\cc}=
\cc^+=\cc \cup t \cc$. Furthermore, $\tchi_0$ is the preferred extension 
of $\chi_0$; in particular, $\tchi_0(1)=\chi_0(1)$. Now, if $\CC \subseteq 
W_n$, then $\CC \cap \tilde{\cc}= \CC \cap \cc$ and $\CC \cap 
\tilde{\Gamma}= \CC \cap \Gamma$. So the assertion holds by ($*$).

On the other hand, if $\CC \subseteq tW_n$, then $\Cd:=
\{t\sigma \mid \sigma\in \CC\}$ is a $\diamond$-conjugacy 
class of $\diamond$-twisted involutions in $W_n$. In this case, we have
$\CC \cap \tilde{\cc}=t(\Cd\cap \cc)$ and $\CC \cap \tilde{\Gamma}=
t(\Cd \cap \Gamma)$. So the assertion holds by Proposition~\ref{invDn}(b). 

{\em Case 2}. Assume that $\cc^\diamond\neq \cc$. Then $\tilde{\cc}=\cc^+=
\cc \cup t\cc \cup \cc t \cup t\cc t$. If $\CC \subseteq tW_n$, then both 
sides of the desired identity are zero. (Indeed, we have $\cc^{-1}=\cc$ 
by \cite[5.2(iii)]{LuB} and so, if $1=(tw)^2=twtw=w^\diamond w$, then 
$w^\diamond=w^{-1} \in \cc \cap \cc^\diamond$, a contradiction.) Now 
assume that $\CC \subseteq W_n$. Then 
\[ |\CC \cap \tilde{\cc}|=|\CC \cap \cc|+|\CC \cap t\cc t|=2\, 
|\CC\cap \cc|.\]  
By ($*$), we have $|\CC\cap\cc|=\chi_0(1)|\CC\cap\Gamma|$ 
and so $|\CC\cap \cc|=\chi_0(1)|\CC\cap\tilde{\Gamma}|$. It remains to note 
that, since $\cc^\diamond \neq \cc$, we have $\chi_0^\diamond\neq \chi_0$ 
and so $\tchi_0$ is obtained by inducing $\chi_0$ from $W_n$ to $\tW_n$;
in particular, $\tchi_0(1)=2\chi_0(1)$. 
\end{proof}

Let $\tilde{\cc}$ be a two-sided $L$-cell of $\tW_n$. We say
that ``{\it Kottwitz' Modified Conjecture holds for $\tilde{\cc}$}'' if, 
for any conjugacy class of involutions $\CC$ in $\tW_n$, we have
\[ \langle \tilde{\kott}_{\CC},[\tilde{\Gamma}]_1\rangle_{\tW_n}=
|\CC \cap \tilde{\Gamma}| \qquad \mbox{for all left $L$-cells 
$\tilde{\Gamma} \subseteq \tilde{\cc}$}.\]
The following remark, together with (\ref{jind}), will provide the basis 
for an inductive proof of Theorem~\ref{maint} where we proceed one
two-sided $L$-cell at a time.

\begin{abschnitt} \label{tense} Let $\tilde{w}_0 \in \tW_n$ be the longest 
element. Let $\CC$ be a conjugacy class of involutions of $\tW_n$. Let 
$\tilde{M}$ be the corresponding $\Q[\tW_n]$-module as in Lemma~\ref{typeD2}.
Since $\tilde{w}_0$ is central in $\tW_n$, the set $\CC \tilde{w}_0$ also 
is a conjugacy class of involutions in $\tW_n$. Let $\tilde{M}_0$ be the 
corresponding $\Q[\tW_n]$-module. Then we have 
\begin{equation*}
\tilde{M}_0\cong\tilde{M}\otimes\tilde{\varepsilon}' \qquad \mbox{and} 
\qquad \tilde{\kott}_{\CC \tilde{w}_0}=\tilde{\kott}_{\CC}\otimes
\varepsilon',\tag{a}
\end{equation*}
where $\tilde{\varepsilon}' \colon W_n \rightarrow \{\pm 1\}$ is the
homomorphism given by $\tilde{\varepsilon}'(t)=1$ and $\tilde{\varepsilon}'
(s_i)=-1$ for $1\leq i \leq n-1$. This follows by an argument entirely 
analogous to that in \cite[Lemma~5.2]{boge}. Now let $\tilde{\Gamma}$ be
a left $L$-cell of $\tW_n$. Then $\tilde{\Gamma}\tilde{w}_0$
also is a left $L$-cell of $\tW_n$; see \cite[11.7]{lusztig}. 
We show that 
\begin{equation*}
[\tilde{\Gamma}\tilde{w}_0]_1=[\tilde{\Gamma}_1] \otimes\varepsilon'.\tag{b}
\end{equation*}
Indeed, by (\ref{exthec1})(b), $[\tilde{\Gamma}]_1$ is obtained by inducing 
$[\Gamma]_1$ from $W_n$ to $\tW_n$ where $\Gamma$ is a left cell of $W_n$ 
such that $\tilde{\Gamma}=\Gamma \cup t\Gamma$. If $n$ is even, then 
$\tilde{w}_0\in W_n$ and this is the longest element in $W_n$. So 
$\tilde{\Gamma}\tilde{w}_0=(\Gamma \tilde{w}_0) \cup t(\Gamma \tilde{w}_0)$.
By \cite[Lemma~5.14]{LuB}, we have  $[\Gamma \tilde{w}_0]_1=[\Gamma]_1 
\otimes \varepsilon$ where $\varepsilon$ is the sign character of $W_n$. 
Since $\varepsilon$ is the restriction of $\varepsilon'$ to $W_n$, this 
implies (b) in this case. Now assume that $n$ is odd. Then $w_0:=t
\tilde{w}_0 \in W_n$ and this is the longest element in $W_n$; furthermore, 
$\Gamma=\Gamma^\diamond=t\Gamma t$. So we obtain $\tilde{\Gamma}\tilde{w}_0
=\Gamma w_0 \cup t(\Gamma w_0)$. Using once more \cite[Lemma~5.14]{LuB}, this 
implies (b) in this case as well. We conclude that 
\begin{equation*}
\langle \tilde{\kott}_{\CC},[\tilde{\Gamma}]_1\rangle_{\tW_n}=
\langle \tilde{\kott}_{\CC \tilde{w}_0},[\tilde{\Gamma}\tilde{w}_0]_1
\rangle_{\tW_n} \qquad \mbox{for any left $L$-cell $\tilde{\Gamma}$
of $\tW_n$}.\tag{c}
\end{equation*}
In particular, Kottwitz' Modified Conjecture holds for a two-sided 
$L$-cell $\tilde{\cc}$ if and only if it holds for the two-sided 
$L$-cell $\tilde{\cc} \tilde{w}_0$.
\end{abschnitt}

Next, there is a standard combinatorial procedure by which certain arguments
about two-sided cells can be reduced to so-called ``cuspidal'' two-sided
cells. This appears and is used at various places in Lusztig's work; 
see, for example, \cite[8.1]{LuB} and \cite[17.13]{cs4}. We have also
used it in \cite[\S 6]{boge} to deal with Kottwitz' conjecture in type
$B_n$. Let us explicitly describe this procedure in our present context.

\begin{abschnitt} \label{jind} For any $r \in \{0,1,\ldots,n\}$, we have 
a parabolic subgroup $\tW_{n-r}=\langle t,s_1,\ldots,s_{n-r-1}\rangle
\subseteq \tW_n$ of type $B_{n-r}$ (where $\tW_0=\{1\}$ and $\tW_1=\{1,t\}$).
Following Lusztig \cite[1.8]{LuDn}, \cite[4.1]{LuB}, we define an additive map
\[ \Jrm_r\colon \Z[\Irr(\tW_{n-r})] \rightarrow \Z[\Irr(\tW_n)]\]
as follows. Consider the parabolic subgroup $\tW'=\tW_{n-r} \times H_r 
\subseteq \tW_n$ where $H_r=\langle s_{n-r+1},\ldots,s_{n-1} \rangle\cong
\SG_r$. Let $\tchi'\in \Irr(\tW_{n-r})$ and $\varepsilon_r$
be the sign character on $H_r$. Since $\ab_{\epsilon_r}=r(r-1)/2$, we have
the implication 
\[ \Big\langle \Ind_{W'}^{\tW_n}\bigl(\tchi' \boxtimes \varepsilon_r\bigr),
\tchi\Big\rangle_{\tW_n} \neq 0 \qquad \Rightarrow \qquad \ab_{\tchi}
\geq \ab_{\tchi'}+r(r-1)/2\]
for all $\tchi \in \Irr(\tW_n)$. We set 
\[ \Jrm_r(\tchi'):=\sum_{\tchi} \Big\langle \Ind_{W'}^{\tW_n}\bigl(
\tchi' \boxtimes \varepsilon_r\bigr),\tchi\Big\rangle_{\tW_n}\, \tchi\]
where the sum runs over all $\tchi \in \Irr(\tW_n)$ such that
$\ab_{\tchi}=\ab_{\tchi'}+r(r-1)/2$. Now let $\tilde{\cc}$ be a two-sided 
$L$-cell of $\tW_n$. Following \cite[8.1]{LuB}, \cite[17.13]{cs4},
we say that $\tilde{\cc}$ is ``{\em smoothly induced}'' if there exists
some $r \in \{1,\ldots,n\}$ and a two-sided cell $L$-cell
$\tilde{\cc}'$ of $\tW_{n-r}$ such that $\Jrm_r$ establishes a bijection
\begin{equation*}
\Irr(\tW_{n-r}\mid \tilde{\cc}') \rightarrow \Irr(\tW_n\mid \tilde{\cc}), 
\qquad \tchi' \mapsto \Jrm_r(\tchi').\tag{a}
\end{equation*}
As in \cite[Remark~6.2]{boge}, one easily sees that then the following holds:
\begin{equation*}
c(\tchi_0)=c(\tchi_0'), \qquad 
d_0(\tchi_0)=d_0(\tchi_0')+\lfloor r/2\rfloor \qquad \mbox{and} \qquad 
|\beta|=|\beta'|+\lfloor r/2\rfloor, \tag{b}
\end{equation*}
where $\tchi_0=\tchi^{(\alpha,\beta)}\in \Irr(\tW_n\mid \tilde{\cc})$ and 
$\tchi_0'=\tchi^{(\alpha',\beta')}\in \Irr( \tW_{n-r}\mid \tilde{\cc}')$ 
are the unique $\diamond$-special characters. Now $\{w_{0,r}\}$ is a 
two-sided cell in $H_r$ where $w_{0,r}\in H_r$ is the longest element; 
furthermore, $\Irr(H_r\mid \{w_{0,r}\})=\{\varepsilon_r\}$. Consequently, 
$\tilde{\cc}'w_{0,r} \subseteq \tW'$ is a two-sided $L$-cell (with respect
to the restriction of $L$ to $\tW'$); thus, using also 
\cite[43.11(b)]{disc10}, we have 
\begin{equation*}
\tilde{\cc}'w_{0,r} \subseteq \tilde{\cc} \qquad \mbox{and} \qquad 
\Irr(\tW'\mid \tilde{\cc}'w_{0,r})=\{ \tchi' \boxtimes \varepsilon_r
\mid \tchi' \in \Irr(\tW_{n-r}\mid \tilde{\cc}')\}.\tag{c}
\end{equation*}
Finally, the point of these definitions is that every two-sided $L$-cell 
$\tilde{\cc}$ of $\tW_n$ is either itself smoothly induced, or the two-sided 
$L$-cell $\tilde{\cc}\tilde{w}_0$ is smoothly induced (where $\tilde{w}_0\in 
\tW_n$ is the longest element), or $n=d^2$ for some $d \geq 2$ in which 
case $\tilde{\cc}$ is uniquely determined; in the last case, $\tilde{\cc}$ 
is called ``cuspidal'' and determined by the condition that the unique 
$\diamond$-special character in $\Irr(\tW_n\mid \tilde{\cc})$ is 
$\tchi^{(\alpha,\beta)}$ where $\alpha=(1,2,\ldots,d-1)$ and $\beta=
(0,1,2,\ldots,d-2)$; see \cite[3.17]{LuDn}, \cite[8.1]{LuB}. 
\end{abschnitt}

\begin{abschnitt} \label{prooft} {\bf Proof of Theorem~\ref{maint}.}
We proceed by induction on $n$. If $n=0$, then $\tW_0=\{1\}$ and the 
assertion is clear. Now assume that $n \geq 1$. Let $\tilde{\cc}$ be a 
two-sided $L$-cell of $\tW_n$. Assume first that $\tilde{\cc}$ is smoothly 
induced; let $r \in \{1,\ldots,n\}$ and $\tilde{\cc}'\subseteq \tW_{n-r}$ 
be as in (\ref{jind}). Let $\CC$ be any conjugacy class of involutions in 
$\tW_n$. Assume that $\sigma_{l,j} \in \CC$ where $l,j \geq 0$ are such 
that $l+2j \leq n$; we write $\CC=\CC_{l,j}$. Let $\tchi_0=\tchi^{(\alpha,
\beta)}\in \Irr(\tW_n \mid \tilde{\cc})$ be the unique $\diamond$-special 
character. Let $\tilde{\Gamma}$ be a left $L$-cell of $\tW_n$ such that 
$\tilde{\Gamma} \subseteq \tilde{\cc}$. By Example~\ref{invD0}(d), we have 
$\langle [\tilde{\Gamma}]_1, \tchi_0 \rangle_{\tW_n}=1$. Hence, using
Proposition~\ref{kottD2}, we already know that 
\[ \langle \tilde{\kott}_{\CC_{l,j}}, [\tilde{\Gamma}]_1\rangle_{\tW_n}= 
\left\{\begin{array}{cl} \displaystyle\binom{c(\alpha,\beta)}{j+l-
d_0(\alpha,\beta)} & \qquad \mbox{if $|\beta|=j$},\\[5mm] 0 & \qquad 
\mbox{otherwise}.\end{array}\right.\]
We will now show that 
\begin{equation*}
|\CC_{l,j} \cap \tilde{\Gamma}| \leq \langle \tilde{\kott}_{\CC_{l,j}}, 
[\tilde{\Gamma}]_1\rangle_{\tW_n}.\tag{$\dagger$}
\end{equation*}
This is seen as follows. If $\CC_{l,j} \cap \tilde{\Gamma}=\varnothing$, 
this is clear. Now assume that $\CC_{l,j}\cap\tilde{\Gamma}\neq\varnothing$.
By Lemma~\ref{invDn1}, the cardinality $|\CC_{l,j} \cap \tilde{\Gamma}|$ does 
not depend on the left $L$-cell $\tilde{\Gamma} \subseteq 
\tilde{\cc}$. Thus, it will be enough to prove ($\dagger$) for one 
particular left $L$-cell in $\tilde{\cc}$. We will choose such a 
left $L$-cell as follows. Consider the two-sided $L$-cell
$\tilde{\cc}'$ of $\tW_{n-r}$. As in (\ref{jind}), let $w_{0,r} \in H_r$ 
be the longest element. Then $\tilde{\cc}'w_{0,r}$ is a two-sided
$L$-cell in $\tW'=\tW_{n-r} \times H_r$ and we have $\tilde{\cc}'w_{0,r}
\subseteq \tilde{\cc}$; see (\ref{jind})(c). Now let $\tilde{\Gamma}'$ be 
a left $L$-cell of $\tW_{n-r}$ which is contained in $\tilde{\cc}$. Then 
$\tilde{\Gamma}'w_{0,r}$ is a left $L$-cell of $\tW'$. Let $\tilde{\Gamma}
\subseteq \tilde{\cc}$ be the left $L$-cell of $\tW_n$ such that
\[ \tilde{\Gamma}'w_{0,r} \subseteq \tilde{\Gamma}.\]
By Example~\ref{invD0}(d) and (\ref{jind})(b), the left $L$-cells
$\tilde{\Gamma}'$, $\tilde{\Gamma}'w_{0,r}$ and $\tilde{\Gamma}$ all contain
the same number of involutions. Hence, all the involutions in 
$\tilde{\Gamma}$ are already contained in $\tilde{\Gamma}'w_{0,r}\subseteq
\tW'$. Consequently, we have
\[ \CC_{l,j} \cap \tilde{\Gamma}=(\CC_{l,j} \cap \tW') \cap 
\tilde{\Gamma}'w_{0,r}.\]
By an argument entirely analogous to that in \cite[Theorem~6.3]{boge} (see 
the paragraph following ($\triangle$) in the proof theoreof), the assumption 
that $\CC_{l,j} \cap \tilde{\Gamma}\neq \varnothing$ now implies that 
$\CC_{l,j} \cap \tW'$ is the conjugacy class containing the element
$\sigma_{l,j'}w_{0,r}$ where $j=j'+k$ and $k=\lfloor r/2 \rfloor$. We 
conclude that
\[ |\CC_{l,j} \cap \tilde{\Gamma}|=|\CC'\cap \tilde{\Gamma}'| \qquad 
\mbox{where $\CC'\subseteq \tW_{n-r}$ is the conjugacy class 
containing $\sigma_{l,j'}$}.\]
Hence, using the equality $j'=j-\lfloor r/2\rfloor$ and (\ref{jind})(b), 
we see that ($\dagger$) holds by induction. Once this is established, it 
actually follows that we must have equality in ($\dagger$). Indeed, 
by Example~\ref{invD0}(d), we have
\[ \sum_{l,j} |\CC_{l,j} \cap \tilde{\Gamma}|\,=\mbox{ (number of 
involutions in $\tilde{\Gamma}$)}=2^{c(\tchi_0)}\]
where the sum runs over all $l,j\geq 0$ such that $l+2j\leq n$. But we obtain
the same result when we sum the binomial coefficients giving the right
hand side of $(\dagger$) over all $l,j$ as above. Hence, all the
inequalities in $(\dagger)$ must be equalities. Thus, Kottwitz' Modified 
Conjecture holds for $\tilde{\cc}$. Then (\ref{tense}) shows that 
Kottwitz' Modified Conjecture also holds for $\tilde{\cc}\tilde{w}_0$. 
By the remarks at the end of (\ref{jind}), these arguments cover all
non-cuspidal two-sided $L$-cells of $\tW_n$. So it remains to show
that Kottwitz' Modified Conjecture holds for the unique cuspidal 
two-sided $L$-cell of $\tW_n$ where $n=d^2$ for some $d \geq 2$. But this
follows from a formal argument based on Lemma~\ref{invDn1}, exactly as 
in the proof of \cite[Theorem~6.3]{boge}.  \qed
\end{abschnitt}

\noindent {\bf Acknowledgements.} I am indebted to R. E. Kottwitz for 
providing me with his personal notes containing further details about the
character-theoretic computations in \cite{kottwitz}. I also wish to thank
G. Pfeiffer for a useful discussion on normalisers of parabolic subgroups.
 

\end{document}